\documentclass[a4paper,twoside]{article}
\usepackage{a4}
\usepackage{amssymb}
\usepackage{amsmath}
\usepackage{upref}
\usepackage[pagebackref,colorlinks,citecolor=blue,linkcolor=blue,urlcolor=blue]{hyperref}
\usepackage[active]{srcltx}
\allowdisplaybreaks[2] % To make it possible for displaybreaks
\newcount\minutes \newcount\hours
\hours=\time \divide\hours 60 \minutes=\hours
 \multiply\minutes-60
\advance\minutes \time
\newcommand{\klockan}{\the\hours:{\ifnum\minutes<10 0\fi}\the\minutes}
\newcommand{\tid}{\today\ \klockan}
\newcommand{\prtid}{\smash{\raise 10mm \hbox{\LaTeX ed \tid}}}
\renewcommand{\prtid}{}
\makeatletter  \headheight 10pt
\def\sectionmark#1{} %\markboth{{\sectnr #1}}{{\sectnr #1}}} %Journal
\def\subsectionmark#1{}
\newcommand{\sectnr}{\ifnum \c@secnumdepth >\z@
                 \thesection.\hskip 1em\relax \fi}
\def\@evenhead{\footnotesize\rm\thepage\hfil\leftmark\hfil\llap{\prtid}}
\def\@oddhead{\footnotesize\rm\rlap{\prtid}\hfil\rightmark\hfil\thepage}
\def\tableofcontents{\section*{Contents} %\@mkboth{Contents}{Contents}} %Journal
 \@starttoc{toc}}
\makeatother
\makeatletter
\def\@biblabel#1{#1.}
\makeatother
\makeatletter
\let\Thebibliography=\thebibliography
\renewcommand{\thebibliography}[1]{\def\@mkboth##1##2{}\Thebibliography{#1}
\addcontentsline{toc}{section}{References}
\frenchspacing % Maybe not needed
\setlength{\@topsep}{0pt}% Delete if extra space before list
\setlength{\itemsep}{0pt}%
\setlength{\parskip}{0pt plus 2pt}%
}
\makeatother
\makeatletter
\def\mdots@{\mathinner.\nonscript\!.%
 \ifx\next,.\else\ifx\next;.\else\ifx\next..\else
 \nonscript\!\mathinner.\fi\fi\fi}
\let\ldots\mdots@
\let\cdots\mdots@
\let\dotso\mdots@
\let\dotsb\mdots@
\let\dotsm\mdots@
\let\dotsc\mdots@
\def\vdots{\vbox{\baselineskip2.8\p@ \lineskiplimit\z@
    \kern6\p@\hbox{.}\hbox{.}\hbox{.}\kern3\p@}}
\def\ddots{\mathinner{\mkern1mu\raise8.6\p@\vbox{\kern7\p@\hbox{.}}%
    \raise5.8\p@\hbox{.}\raise3\p@\hbox{.}\mkern1mu}}
\makeatother
\makeatletter
\let\Enumerate=\enumerate
\renewcommand{\enumerate}{\Enumerate%
\setlength{\@topsep}{0pt}% Delete if extra space before list
\setlength{\itemsep}{0pt}%
\setlength{\parskip}{0pt plus 1pt}%
\renewcommand{\theenumi}{\textup{(\alph{enumi})}}%
\renewcommand{\labelenumi}{\theenumi}%
}
\let\endEnumerate=\endenumerate
\renewcommand{\endenumerate}{\endEnumerate\ifhmode\unskip\fi}
\makeatother
\makeatletter
\def\@seccntformat#1{\csname the#1\endcsname.\quad}
\makeatother
\newcommand{\authortitle}[2]{\author{#1}\title{#2}\markboth{#1}{#2}}
\newcommand{\art}[6]{{\sc #1, \rm #2, \it #3 \bf #4 \rm (#5), \mbox{#6}.}}
\newcommand{\auth}[2]{{#1, #2.}}
\newcommand{\artin}[3]{{\sc #1, \rm #2,  in #3.}}

\newcommand{\book}[3]{{\sc #1, \it #2, \rm #3.}}
\newcommand{\AND}{{\rm and }}
\RequirePackage{amsthm}
\newtheoremstyle{descriptive}%
  {\topsep}   %{\medskipamount}          % Space above
  {\topsep}   %  {\medskipamount}          % Space below
  {\rmfamily} % Body font
  {}          % Indent
  {\bfseries} % Head font
  {.}         % Punctuation after thm head
  { }         % Space after thm head
  {}          % Thm head spec(?)
\newtheoremstyle{propositional}%
  {\topsep}   %  {\medskipamount}          % Space above
  {\topsep}   %  {\medskipamount}          % Space below
  {\itshape}  % Body font
  {}          % Indent
  {\bfseries} % Head font
  {.}         % Punctuation after thm head
  { }         % Space after thm head
  {}          % Thm head spec(?)
\theoremstyle{propositional}
\newtheorem{thm}{Theorem}[section]
\newtheorem{prop}[thm]{Proposition}
\newtheorem{lem}[thm]{Lemma}
\newtheorem{cor}[thm]{Corollary}
\theoremstyle{descriptive}
\newtheorem{deff}[thm]{Definition}
\newtheorem{example}[thm]{Example}
\newtheorem{remark}[thm]{Remark}
\newtheorem{openprobs}[thm]{Open problems}
\makeatletter
\renewenvironment{proof}[1][\proofname]{\par
  \pushQED{\qed}%
  \normalfont
  \trivlist
  \item[\hskip\labelsep
        \itshape
    #1\@addpunct{.}]\ignorespaces
}{%
  \popQED\endtrivlist\@endpefalse
} \makeatother
\def\vint{\mathop{\mathchoice%
          {\setbox0\hbox{$\displaystyle\intop$}\kern 0.22\wd0%
           \vcenter{\hrule width 0.6\wd0}\kern -0.82\wd0}%
          {\setbox0\hbox{$\textstyle\intop$}\kern 0.2\wd0%
           \vcenter{\hrule width 0.6\wd0}\kern -0.8\wd0}%
          {\setbox0\hbox{$\scriptstyle\intop$}\kern 0.2\wd0%
           \vcenter{\hrule width 0.6\wd0}\kern -0.8\wd0}%
          {\setbox0\hbox{$\scriptscriptstyle\intop$}\kern 0.2\wd0%
           \vcenter{\hrule width 0.6\wd0}\kern -0.8\wd0}}%
          \mathopen{}\int}
{\catcode`p =12 \catcode`t =12 \gdef\eeaa#1pt{#1}}      % Get slantfactor
\def\accentadjtext#1{\setbox0\hbox{$#1$}\kern   % Convert it with height
                \expandafter\eeaa\the\fontdimen1\textfont1 \ht0 }
\def\accentadjscript#1{\setbox0\hbox{$#1$}\kern % Convert it with height
                \expandafter\eeaa\the\fontdimen1\scriptfont1 \ht0 }
\def\accentadjscriptscript#1{\setbox0\hbox{$#1$}\kern   % Convert it with height
                \expandafter\eeaa\the\fontdimen1\scriptscriptfont1 \ht0 }
\def\accentadjtextback#1{\setbox0\hbox{$#1$}\kern       % Convert it with height
                -\expandafter\eeaa\the\fontdimen1\textfont1 \ht0 }
\def\accentadjscriptback#1{\setbox0\hbox{$#1$}\kern     % Convert it with height
                -\expandafter\eeaa\the\fontdimen1\scriptfont1 \ht0 }
\def\accentadjscriptscriptback#1{\setbox0\hbox{$#1$}\kern % Convert it with height
                -\expandafter\eeaa\the\fontdimen1\scriptscriptfont1 \ht0 }
\def\itoverline#1{{\mathsurround0pt\mathchoice
        {\rlap{$\accentadjtext{\displaystyle #1}
                \accentadjtext{\vrule height1.593pt}
                \overline{\phantom{\displaystyle #1}
                \accentadjtextback{\displaystyle #1}}$}{#1}}
        {\rlap{$\accentadjtext{\textstyle #1}
                \accentadjtext{\vrule height1.593pt}
                \overline{\phantom{\textstyle #1}
                \accentadjtextback{\textstyle #1}}$}{#1}}
        {\rlap{$\accentadjscript{\scriptstyle #1}
                \accentadjscript{\vrule height1.593pt}
                \overline{\phantom{\scriptstyle #1}
                \accentadjscriptback{\scriptstyle #1}}$}{#1}}
        {\rlap{$\accentadjscriptscript{\scriptscriptstyle #1}
                \accentadjscriptscript{\vrule height1.593pt}
                \overline{\phantom{\scriptscriptstyle #1}
                \accentadjscriptscriptback{\scriptscriptstyle #1}}$}{#1}}}}
\newcommand{\limplus}{{\mathchoice{\vcenter{\hbox{$\scriptstyle +$}}}
  {\vcenter{\hbox{$\scriptstyle +$}}}
  {\vcenter{\hbox{$\scriptscriptstyle +$}}}
  {\vcenter{\hbox{$\scriptscriptstyle +$}}}
}}
\newcommand{\limminus}{{\mathchoice{\vcenter{\hbox{$\scriptstyle -$}}}
  {\vcenter{\hbox{$\scriptstyle -$}}}
  {\vcenter{\hbox{$\scriptscriptstyle -$}}}
  {\vcenter{\hbox{$\scriptscriptstyle -$}}}
}}
\newcommand{\limpm}{{\mathchoice{\vcenter{\hbox{$\scriptstyle \pm$}}}
  {\vcenter{\hbox{$\scriptstyle \pm$}}}
  {\vcenter{\hbox{$\scriptscriptstyle \pm$}}}
  {\vcenter{\hbox{$\scriptscriptstyle \pm$}}}
}}
\newcommand{\setm}{\setminus}
\renewcommand{\emptyset}{\varnothing}
\newcommand{\Cp}{{C_p}}
\newcommand{\CpU}{{C_p^U}}
\newcommand{\CpE}{{C_p^E}}
\newcommand{\CpX}{{C_p^X}}
\DeclareMathOperator{\capp}{cap}
\newcommand{\cp}{\capp_p}
\DeclareMathOperator{\diam}{diam}
\DeclareMathOperator{\diverg}{div}
\DeclareMathOperator{\Lip}{Lip}
\newcommand{\Lipc}{{\Lip_c}}
\DeclareMathOperator{\spt}{supp}
\newcommand{\supp}{\spt}
\DeclareMathOperator*{\essinf}{ess\,inf}
\DeclareMathOperator*{\finelimsup}{fine\,lim\,sup}
\DeclareMathOperator*{\fineliminf}{fine\,lim\,inf}
\DeclareMathOperator*{\finelim}{fine\,lim}
\DeclareMathOperator*{\esslim}{ess\,lim}
\DeclareMathOperator*{\essliminf}{ess\,lim\,inf}
\DeclareMathOperator*{\esslimsup}{ess\,lim\,sup}
\DeclareMathOperator{\fineint}{fine-int}
\DeclareMathOperator*{\cpessinfalt}{\text{$\Cp$-}\essinf}
\newcommand{\cplimsup}{\text{$\Cp$-}\esslimsup}
\newcommand{\cpliminf}{\text{$\Cp$-}\essliminf}
\newcommand{\cplim}{\text{$\Cp$-}\esslim}
\newcommand{\bdry}{\partial}
\newcommand{\bdy}{\bdry}
\newcommand{\loc}{_{\rm loc}}
\newcommand{\fineloc}{_\textup{fine-loc}}
\newcommand{\eps}{\varepsilon}
\newcommand{\ga}{\gamma}
\newcommand{\Ga}{\Gamma}
\newcommand{\Om}{\Omega}
\renewcommand{\phi}{\varphi}
\newcommand{\p}{{$p\mspace{1mu}$}}
\newcommand{\clEp}{{\itoverline{E}\mspace{1mu}}^p}
\newcommand{\clUp}{{\overline{U}\mspace{1mu}}^p}
\newcommand{\bdyp}{\bdy_p} % Fine boundary
\newcommand{\R}{\mathbf{R}}
\newcommand{\eR}{{\overline{\R}}}
\newcommand{\K}{{\cal K}}
\newcommand{\Np}{N^{1,p}}
\newcommand{\Nploc}{N^{1,p}\loc}
\newcommand{\Wploc}{W^{1,p}\loc}
\newcommand{\Npploc}{N^{1,p}_{\textup{fine-loc}}}
\newcommand{\Lploc}{L^{p}\loc}
\newcommand{\ut}{\tilde{u}}
\newcommand{\ft}{\tilde{f}}
\def\cprime{{\mathsurround0pt$'$}}
\numberwithin{equation}{section}
\newcommand{\eqv}{\ensuremath{
\mathchoice{\quad \Longleftrightarrow \quad}{\Leftrightarrow}
                {\Leftrightarrow}{\Leftrightarrow}} }
\newcommand{\imp}{\ensuremath{\Rightarrow} }

\newenvironment{ack}{\medskip{\it Acknowledgement.}}{}

\begin{document}

\authortitle{Anders Bj\"orn, Jana Bj\"orn and Visa Latvala}
{The Dirichlet problem for \p-minimizers on finely open sets in metric spaces}

\author{
Anders Bj\"orn \\
\it\small Department of Mathematics, Link\"oping University, SE-581 83 Link\"oping, Sweden\\
\it \small anders.bjorn@liu.se, ORCID\/\textup{:} 0000-0002-9677-8321
\\
\\
Jana Bj\"orn \\
\it\small Department of Mathematics, Link\"oping University, SE-581 83 Link\"oping, Sweden\\
\it \small jana.bjorn@liu.se, ORCID\/\textup{:} 0000-0002-1238-6751
\\
\\
Visa Latvala \\
\it\small Department of Physics and Mathematics,
University of Eastern Finland, \\
\it\small  P.O. Box 111, FI-80101 Joensuu,
Finland\/{\rm ;} \\
\it \small visa.latvala@uef.fi, ORCID\/\textup{:} 0000-0001-9275-7331
}

\date{Preliminary version, \today}
\date{}

\maketitle

\noindent{\small
  {\bf Abstract}. We initiate the study of fine \p-(super)minimizers, associated
  with \p-harmonic functions, on finely open sets in metric spaces,
   where $1 < p < \infty$.
After having developed their basic theory,
we obtain
the \p-fine continuity of the solution of the Dirichlet problem on
a finely open set with continuous Sobolev boundary values,
as a by-product of similar pointwise results.
These results are new also on unweighted $\R^n$.
We build this theory in
a complete metric space equipped with a doubling measure
supporting a \p-Poincar\'e inequality.
}

\medskip

\medskip

\noindent {\small Mathematics Subject Classification (2020):
Primary:
31E05;  % Potential theory on fractals and metric spaces
 Secondary:
30L99, % (2010-now)  Analysis on metric spaces,  None of the above, but in this section
31C40, % Fine potential theory; fine properties of sets and functions
35J92. % (2010-now) Quasilinear elliptic equations with p-Laplacian
}

\section{Introduction}

Superharmonic functions play a fundamental role in the classical
potential theory.
Unlike harmonic functions (i.e.\ solutions of the Laplace equation
$\Delta u=0$), they need not be continuous but are finely
continuous.
In fact, the fine topology is the coarsest topology that
makes all superharmonic functions continuous, see Cartan~\cite{cartan46}.
The fine topology is closely related to the Dirichlet
boundary value problem
for the Laplace equation on open sets.
It follows from the famous
Wiener criterion~\cite{wiener} that a
boundary point $x_0\in\bdy\Om$ of a Euclidean domain $\Om$  is irregular for
$\Delta u=0$ if and only if $\Om\cup\{x_0\}$ is finely open, i.e.\
if the complement  $\R^n\setm\Om$ is thin at $x_0$ in a capacity
density sense. In this case, the complement and the boundary are simply too small in the potential
theoretical sense to ensure that continuous boundary data enforce
continuity of the corresponding solution at $x_0$.
These facts have lead to the development of fine potential theory and
finely (super)harmonic functions associated with $\Delta u=0$ on finely open sets,
see the monograph \cite{Fug} of Fuglede,
the papers \cite{Fug74}--\cite{Fug90},
\cite{LuMa}, \cite{Lyons80:1}, \cite{Lyons80:2},  and the book~\cite{LuMaZa}
  by Luke\v{s}--Mal\'y--Zaj\'i\v{c}ek,
which contain additional results and references.

In the nonlinear case, for equations associated with the \p-Laplacian
$\Delta_p$ and $1<p \ne 2$, the first similar study  was
conducted by Kilpel\"ainen--Mal\'y~\cite{KiMa92},
who studied \p-fine (super)solutions for such equations
on quasiopen subsets of unweighted $\R^n$.
That theory was further extended by Latvala~\cite{LatPhD}, \cite{Lat00}, in
particular for $p=n$.
Eigenvalue problems for the \p-Laplacian in quasiopen subsets of
$\R^n$ were  considered in Fusco--Mukherjee--Zhang~\cite{FuscoMZ}.
We are not  aware of any other papers
dealing with \p-fine (super)solutions, and in particular
none beyond unweighted $\R^n$.

The Wiener criterion was extended to the nonlinear theory
associated with
\p-harmonic functions on subsets of
(unweighted and weighted)
$\R^n$ in \cite{HeKiMa}, \cite{KiMa94}, \cite{Lind-Mar}, \cite{Maz70}
and \cite{Mikkonen}, and partially also to metric spaces,
  see \cite{JB-Matsue}--\cite{BMS}.
It has also been related to fine continuity of \p-superharmonic
functions on \emph{open} sets in much the same way as for the Laplacian.
Following this nonlinear development,  we define the fine topology
on metric spaces using the notion of thinness based on a Wiener type
integral, see Definition~\ref{deff-thinness}.

In this paper we continue our study of fine potential theory on metric
spaces, carried through in~\cite{BBLat1}--\cite{BBLat2},
and initiate the study of fine \p-(super)minimizers
with $1<p< \infty$.
We consider a complete metric space $X$ equipped with a doubling measure
supporting a \p-Poincar\'e inequality.
The function space naturally associated with \p-energy minimizers on
such metric spaces is the Sobolev type space $\Np$, called the Newtonian
space.

The following regularity result for solutions of the Dirichlet problem
on finely open sets is our main result, which we obtain
as a by-product of more general pointwise results.
Even in unweighted $\R^n$ and for $\Delta_pu=0$, it is more
general than the similar Theorem~5.3 in Kilpel\"ainen--Mal\'y~\cite{KiMa92}.

Here $\clUp$ is the fine closure of $U$.

\begin{thm} \label{thm-finecont-dir-intro}
Let $U \subset X$ be finely open and let
either $f \in C(U) \cap \Np(U)$ or
$f \in C(\clUp \cap \bdy U) \cap \Np(X)$, where in both cases $f$ is assumed
to be continuous as a function with values in $\eR:=[-\infty,\infty]$.
Then there is a finely continuous solution
of the Dirichlet problem in $U$ with boundary values $f$.
\end{thm}

In the linear axiomatic setting, i.e.\ for $p=2$, finely (super)harmonic
functions and the Dirichlet problem on finely open sets
have been rigorously investigated,
see the monographs~\cite{Fug} and~\cite{LuMaZa}.
As pointed out in~\cite[p.~389]{LuMaZa}, even in the linear
setting the fine boundary can be too small for a fruitful theory
of the Dirichlet problem. 
Thus the use of the metric boundary in Theorem~\ref{thm-finecont-dir-intro}
is perhaps less unnatural than it may at first seem.

Obviously, some of the linear tools used in~\cite{Fug} and~\cite{LuMaZa}
are not available to us, nor in the nonlinear setting of unweighted
$\R^n$ and $p\ne2$.
Already the notion of fine \p-(super)harmonic functions is not straightforward,
and it is an open question whether \p-(super)minimizers  on finely open
sets have finely continuous representatives.
There are other open problems concerning important
properties of such functions,
see Section~\ref{sect-open-prob} for further discussion.

In metric spaces, there is
(in general) no equation to work with (such as the \p-Laplace equation).
Therefore our theory relies on \p-fine (super)minimizers
defined through \p-energy integrals and
upper gradients. This
makes our approach essentially
independent of the theory in
Kilpel\"ainen--Mal\'y~\cite{KiMa92} and Latvala~\cite{LatPhD}, \cite{Lat00},
even though our main result was inspired by the proof of Theorem~5.3 in~\cite{KiMa92}.
The key arguments in both proofs rely on pasting lemmas and the fine continuity of
\p-superharmonic functions on open sets.

Finely open sets and fine topology are closely related to quasiopen
sets and quasitopology, as shown by Fuglede~\cite{Fugl71}.
A similar study on metric spaces is more recent, but the metric space
approach seems suitable since it makes it easy
to consider the Sobolev type spaces $\Np$
on nonopen sets, such as finely open and quasiopen sets.
These Newtonian spaces were shown in \cite{BBLat3} and \cite{BBMaly} to
coincide with the Sobolev spaces developed on quasiopen and finely open sets
in $\R^n$ by Kilpel\"ainen--Mal\'y~\cite{KiMa92}.
Moreover, functions in the spaces $\Np$ are automatically
quasicontinuous, and consequently finely continuous outside a set of
zero capacity, both on open and quasiopen sets, see \cite{BBLat3},
\cite{BBMaly}, \cite{BBS5}, \cite{JB-pfine} and \cite{korte08}.
Several of these results play a crucial role in this paper.
On unweighted $\R^n$ and for nonlinear fine potential theory,
they can be found
in the monograph by Mal\'y--Ziemer~\cite{MZ}. See also
Heinonen--Kilpel\"ainen--Martio~\cite{HeKiMa}
for many of these results on weighted $\R^n$,
as well as \cite{BBfusco}, \cite{BBMaly}, \cite{Lahti17}--\cite{Lahti20}
for further results.
  
Obstacle problems, and thereby (super)minimizers, on nonopen sets in metric spaces
were studied in \cite{BBnonopen} and  it was shown therein (\cite[Theorem~7.3]{BBnonopen})
that the theory of obstacle problems is not natural beyond
finely open (or quasiopen) sets. In Proposition~\ref{prop:beyond},
we show that this true also for the theory of \p-fine (super)minimizers.
Additional fine properties of (super)harmonic functions on open sets
were derived in \cite{BBLat1} and~\cite{BBLat2}.

In addition to fine potential theory,
quasiopen sets appear naturally as minimizing sets in shape
optimization problems, see e.g.\ Bucur--Buttazzo--Velichkov~\cite{BBV}, Buttazzo--Dal Maso~\cite{buttazzo-dalMaso},
Buttazzo--Shrivastava~\cite[Examples~4.3 and~4.4]{buttazzo-shr}
and Fusco--Mukherjee--Zhang~\cite{FuscoMZ}.
Their importance lies in the fact that
they are sub- and superlevel sets of Sobolev functions,
see Theorem~\ref{thm-quasiopen-char}.

The outline of the paper is as follows:
In Section~\ref{sect-prelim},
we recall some definitions from first-order analysis on metric spaces,
while the fine topology is introduced
in Section~\ref{sect-fine-top}.
Therein, we also give two new characterizations of
quasiopen sets, which are probably known to the experts in the field.

In order to be able to study \p-fine (super)minimizers and the Dirichlet problem
on quasiopen sets $U$,
we need the appropriate local Newtonian (Sobolev) space
$\Npploc(U)$.
We study this space in Section~\ref{sect-fine-strict},
where we also establish a
density result that plays a crucial role in later sections.
In Sections~\ref{sect-fine-min} and~\ref{sect-obst},
we develop the basic theory of
\p-fine (super)minimizers, obstacle and Dirichlet problems on quasiopen sets.

Finally, in Section~\ref{sect-fine-cont},
we are ready to develop
the necessary framework
enabling us to obtain Theorem~\ref{thm-finecont-dir-intro}.
We also deduce corresponding pointwise results.
In Section~\ref{sect-remove},
we use some of our results to
give some more information
on fine Newtonian spaces.
The final Section~\ref{sect-open-prob} is devoted
to open problems.

\begin{ack}
A.~B. and J.~B. were supported by the Swedish Research Council,
grants 2016-03424 and 2020-04011 resp.\ 621-2014-3974 and 2018-04106.
Part of this research was done during
several visits of V.~L. to Link\"oping University.
\end{ack}

\section{Notation and preliminaries}
\label{sect-prelim}

\emph{We assume throughout the paper
that $X=(X,d,\mu)$ is a metric space equipped
with a metric $d$ and a positive complete  Borel  measure $\mu$
such that $0<\mu(B)<\infty$ for all balls $B \subset X$.
We also assume that $1<p< \infty$.}

\medskip

In this section, we introduce
the necessary metric space concepts used in this paper.
For brevity, we refer to
Bj\"orn--Bj\"orn--Latvala~\cite{BBLat1}, \cite{BBLat2} for more
extensive introductions, and references to the literature.
See also the monographs Bj\"orn--Bj\"orn~\cite{BBbook} and
Heinonen--Koskela--Shanmugalingam--Tyson~\cite{HKST},
where the theory is thoroughly  developed with proofs.

The measure  $\mu$  is \emph{doubling} if
there exists $C>0$ such that for all balls
$B=B(x_0,r):=\{x\in X: d(x,x_0)<r\}$ in~$X$,
we have
$        0 < \mu(2B) \le C \mu(B) < \infty$,
where $\lambda B=B(x_0,\lambda r)$.
In this paper, all balls are open.

A \emph{curve} is a continuous mapping from an interval,
and a \emph{rectifiable} curve is a curve with finite length.
We will only consider curves which are nonconstant, compact and rectifiable.
A curve can thus be parameterized by its arc length $ds$.
A property holds for \emph{\p-almost every curve}
if the curve family $\Ga$ for which it fails has zero \p-modulus,
i.e.\ there is $\rho\in L^p(X)$ such that
$\int_\ga \rho\,ds=\infty$ for every $\ga\in\Ga$.

\begin{deff} \label{deff-ug}
A Borel function $g:X \to [0,\infty]$ is a \p-weak \emph{upper gradient}
of $f:X \to \eR:=[-\infty,\infty]$
if for \p-almost all curves
$\gamma: [0,l_{\gamma}] \to X$,
\begin{equation} \label{ug-cond}
        |f(\gamma(0)) - f(\gamma(l_{\gamma}))| \le \int_{\gamma} g\,ds,
\end{equation}
where the left-hand side is $\infty$
whenever at least one of the
terms therein is infinite.
\end{deff}

If $f$ has a \p-weak upper gradient in $\Lploc(X)$, then
it has a \emph{minimal \p-weak upper gradient}
$g_f \in \Lploc(X)$
in the sense that
$g_f \le g$ a.e.\
for every \p-weak upper gradient $g \in \Lploc(X)$ of $f$.

\begin{deff} \label{deff-Np}
Let for measurable $f$,
\[
        \|f\|_{\Np(X)} = \biggl( \int_X |f|^p \, d\mu
                + \inf_g  \int_X g^p \, d\mu \biggr)^{1/p},
\]
where the infimum is taken over all \p-weak upper gradients of $f$.
The \emph{Newtonian space} on $X$ is
\[
        \Np (X) = \{f: \|f\|_{\Np(X)} <\infty \}.
\]
\end{deff}
\medskip

The space $\Np(X)/{\sim}$, where  $f \sim h$ if and only if $\|f-h\|_{\Np(X)}=0$,
is a Banach space and a lattice.
In this paper we assume that functions in $\Np(X)$
 are defined everywhere (with values in $\eR$),
not just up to an equivalence class in the corresponding function space.

For a measurable set $E\subset X$, the Newtonian space $\Np(E)$ is defined by
considering $(E,d|_E,\mu|_E)$ as a metric space in its own right.
We say  that $f \in \Nploc(E)$ if
for every $x \in E$ there exists a ball $B_x\ni x$ such that
$f \in \Np(B_x \cap E)$.
If $f,h \in \Nploc(X)$,
then $g_f=g_h$ a.e.\ in $\{x \in X : f(x)=h(x)\}$,
in particular $g_{\min\{f,c\}}=g_f \chi_{\{f < c\}}$ a.e.\ in $X$
for $c \in \R$.

The  \emph{Sobolev capacity} of an arbitrary set $E\subset X$ is
\[
\Cp(E) = \CpX(E)=\inf_{f}\|f\|_{\Np(X)}^p,
\]
where the infimum is taken over all $f \in \Np(X)$ such that
$f\geq 1$ on $E$.
A property holds \emph{quasieverywhere} (q.e.)\
if the set of points  for which it fails
has capacity zero.
The capacity is the correct gauge
for distinguishing between two Newtonian functions.
If $f \in \Np(X)$, then $h \sim f$ if and only if $h=f$ q.e.
Moreover, if $f,h \in \Np(X)$ and $f= h$ a.e., then $f=h$ q.e.

For $A \subset U \subset X$, where $U$ is assumed to be measurable, we let
\[
  \Np_0(A,U)=\{f|_{A} : f \in \Np(U) \text{ and }
  f=0 \text{ on } U \setm A\}.
\]
If $U=X$, we write $\Np_0(A)=\Np_0(A,X)$.
Functions from $\Np_0(A,U)$ can be extended by zero in $U\setm A$ and we
will regard them in that sense if needed.

If $E \subset A$ are bounded subsets of $X$, then
the \emph{variational capacity} of $E$ with respect to $A$ is
\[
\cp(E,A) = \inf_{f}\int_{X} g_{f}^p\, d\mu,
\]
where the infimum is taken over all $f \in \Np_0(A)$
such that $f\geq 1$ on $E$.
If no such function $f$ exists then $\cp(E,A)=\infty$.

\begin{deff} \label{def-PI}
$X$ supports a \emph{\p-Poincar\'e inequality} if
there exist constants $C>0$ and $\lambda \ge 1$
such that for all balls $B \subset X$,
all integrable functions $f$ on $X$ and all \p-weak upper gradients $g$ of $f$,
\begin{equation} \label{PI-ineq}
        \vint_{B} |f-f_B| \,d\mu
        \le C \diam(B) \biggl( \vint_{\lambda B} g^{p} \,d\mu \biggr)^{1/p},
\end{equation}
where $ f_B
 :=\vint_B f \,d\mu
:= \int_B f\, d\mu/\mu(B)$.
\end{deff}

In $\R^n$ equipped with a doubling measure $d\mu=w\,dx$, where
$dx$ denotes Lebesgue measure, the \p-Poincar\'e inequality~\eqref{PI-ineq}
is equivalent to the \emph{\p-admissibility} of the weight $w$ in the
sense of Heinonen--Kilpel\"ainen--Martio~\cite{HeKiMa}, see
Corollary~20.9 in~\cite{HeKiMa}
and Proposition~A.17 in~\cite{BBbook}.
Moreover, in this case $g_u=|\nabla u|$ a.e.\ if $u \in \Np(\R^n)$.

As usual, we will write $f=f_\limplus-f_\limminus$,
where $f_\limpm=\max\{\pm f,0\}$.

\section{Fine topology and Newtonian functions on finely open sets}
\label{sect-fine-top}

\emph{Throughout the rest of the paper, we assume
that $X$ is complete and  supports a \p-Poincar\'e inequality, that
$\mu$ is doubling, and that $1<p< \infty$.}

\medskip

To avoid pathological situations we also assume that $X$
contains at least two points. In this section we recall the basic facts about the fine topology
associated with Newtonian functions.

\begin{deff}\label{deff-thinness}
A set $E\subset X$ is  \emph{thin} at $x\in X$ if
\begin{equation*} %  \label{deff-thin}
\int_0^1\biggl(\frac{\cp(E\cap B(x,r),B(x,2r))}{\cp(B(x,r),B(x,2r))}\biggr)^{1/(p-1)}
     \frac{dr}{r}<\infty.
\end{equation*}
A set $V\subset X$ is \emph{finely open} if
$X\setminus V$ is thin at each point $x\in V$.
\end{deff}

In the definition of thinness,
we make the convention that the integrand
is 1 whenever $\cp(B(x,r),B(x,2r))=0$.
It is easy to see that the finely open sets give rise to a
topology, which is called the \emph{fine topology}.
Every open set is finely open, but the converse is not true in general.
A function $u : V \to \eR$, defined on a finely open set $V$, is
\emph{finely continuous} if it is continuous when $V$ is equipped with the
fine topology and $\eR$ with the usual topology.
See Bj\"orn--Bj\"orn~\cite[Section~11.6]{BBbook} and
Bj\"orn--Bj\"orn--Latvala~\cite{BBLat1}
for further discussion on thinness and the fine topology
in metric spaces.
The
fine interior, fine boundary and fine closure of $E$
are denoted $\fineint E$, $\bdyp E$ and $\clEp$, respectively.

The following characterization of the fine boundary
is from Corollary~7.8 in  Bj\"orn--Bj\"orn~\cite{BBnonopen}.
We will mainly use it for finely open sets.

\begin{lem} \label{lem-bdyp}
Let $E\subset X$ be arbitrary. Then the fine boundary of $E$ is
\[
\bdyp E = \{x\in E: X\setm E \text{ is not thin at }x\}
     \cup \{x\in X\setm E: E \text{ is not thin at }x\}.
\]
\end{lem}

The following definition will also be important in this paper.

\begin{deff}   \label{deff-q-cont}
A set $U\subset X$ is \emph{quasiopen} if for every
$\varepsilon>0$ there is an open set $G\subset X$ such that $\Cp(G)<\varepsilon$
and $G\cup U$ is open.

A function $u$ defined on a set  $E \subset X$ is \emph{quasicontinuous}
if for every $\varepsilon>0$ there is an open set $G\subset X$
such that $\Cp(G)<\varepsilon$ and $u|_{E \setm G}$ is finite and
continuous.
\end{deff}

The quasiopen sets do not in general form a topology,
see Remark~9.1 in Bj\"orn--Bj\"orn~\cite{BBnonopen}.
However it follows easily from the countable subadditivity of $\Cp$
that countable unions and finite intersections of quasiopen sets are quasiopen.
Quasiopen sets have recently been characterized
in several ways. Here we summarize the known and some new characterizations.
Note in particular the close connection between quasiopen and finely
open sets.

\begin{thm} \label{thm-quasiopen-char}
Let $U\subset X$ be arbitrary.
Then the following conditions are equivalent\/\textup{:}
\begin{enumerate}
\renewcommand{\theenumi}{\textup{(\roman{enumi})}}%
\item \label{b-1}
$U$ is quasiopen\/\textup{;}
\item \label{b-2}
$U$ is a union of a finely open set and a set of capacity zero\/\textup{;}
\item \label{b-p-path}
$U$ is \p-path open,
i.e.\ $\ga^{-1}(U)$ is relatively open in $[0,l_\ga]$
for \p-almost every curve $\gamma: [0,l_{\gamma}] \to X$\/\textup{;}
\item \label{b-3}
$U=\{x:u(x)>0\}$ for some nonnegative quasicontinuous $u$ on
$X$\textup{;}
\item \label{b-4}
$U=\{x:u(x)>0\}$ for some nonnegative $u\in \Np(X)$.
\end{enumerate}
\end{thm}

\begin{proof}
\ref{b-1} \eqv \ref{b-2}
This follows from Theorem~1.4\,(a) in Bj\"orn--Bj\"orn--Latvala~\cite{BBLat2}.

\ref{b-1} \imp \ref{b-p-path}
This follows from
Remark~3.5 in Shanmugalingam~\cite{Sh-harm}.

\ref{b-p-path} \imp \ref{b-1}
This follows from Theorem~1.1 in Bj\"orn--Bj\"orn--Mal\'y~\cite{BBMaly}.

\ref{b-4} \imp \ref{b-3}
By Theorem~1.1 in Bj\"orn--Bj\"orn--Shan\-mu\-ga\-lin\-gam~\cite{BBS5},
$u$ is quasicontinuous and thus \ref{b-3} holds.

\ref{b-3} \imp \ref{b-1}
This follows from Proposition~3.4 in \cite{BBMaly}.

\ref{b-2} \imp \ref{b-4}
Assume that $V\subset U$ is finely open and $\Cp(U\setminus V)=0$.
By
Lemma~3.3 in Bj\"orn--Bj\"orn--Latvala~\cite{BBLat3},
for each $x\in V$
we find $v_x\in \Np_0(V)$ such that $0\le v_x\le 1$, $v_x(x)=1$
and $v_x=0$ outside of $V$.
Since $v_x$ is quasicontinuous
(by Theorem~1.1 in~\cite{BBS5}),
$\{y:v_x(y)>0\}$ is
a quasiopen subset of $V$ (by the already proved \ref{b-3} \imp \ref{b-1}).
Therefore, by the quasi-Lindel\"of principle (Theorem~3.4 in~\cite{BBLat3}),
the family $\{v_x:x\in V\}$ contains a
countable subfamily $\{v_j\}_{j=1}^\infty$
such that $\Cp(Z)=0$ for the set
\[
Z:=\{x\in V: v_j(x)=0 \textrm{ for all } j\}.
\]
Let
\[
v=\sum_{j=1}^\infty \frac{2^{-j}v_j}{1+\|v_j\|_{\Np(X)}} \in \Np(X)
\quad \text{and} \quad
 u=\begin{cases}
  1 & \textrm{on }  Z\cup(U\setminus V),\\
  v & \textrm{elsewhere}.
\end{cases}
\]
Since $u=v$ q.e., also $u \in \Np(X)$. Moreover $U=\{x:u(x)>0\}$.
\end{proof}

Quasiopen, and thus finely open, sets are measurable.
If $U$ is finely open and $\Cp(E)=0$, then $U \setm E$ is
finely open, from which it follows that fine limits do not see sets of capacity zero.

For any measurable set $E\subset X$ the notion of q.e.\ in $E$
can either be taken with respect to the global capacity $\Cp$ on $X$
or with respect to the capacity $\CpE$ determined
by $E$ as the underlying space. However, for a quasiopen set $U$,
the capacities $\Cp$ and $\CpU$ have the same zero sets,
and $\Cp$-quasicontinuity in $U$ is equivalent to
$\CpU$-quasicontinuity, by Propositions~3.4 and~4.2 in
\cite{BBMaly}.

Here we collect some  facts on quasicontinuity from
\cite[Theorem~4.4]{BBLat3},
\cite[Theorem~1.4]{BBLat2} and
\cite[Theorem~1.3]{BBMaly}.
For further characterizations of quasiopen sets and quasicontinuous functions
see~\cite{BBMaly} and also
Theorem~\ref{thm-Np0} below.

\begin{thm} \label{thm-newton-quasicont}
Let $U$ be quasiopen. Then the following are true\/\textup{:}
\begin{enumerate}
\item \label{k-1}
Functions in $\Np(U)$ are quasicontinuous.
\item \label{k-3}
A function $u:U \to \eR$  is quasicontinuous if and only if
it is finite q.e.\ and finely continuous q.e.
\end{enumerate}
\end{thm}

\section{\texorpdfstring{$\Npploc(U)$}{N1,p,fine-loc(U)} and
  \texorpdfstring{\p}{p}-strict subsets}
\label{sect-fine-strict}

\emph{From now on we always assume that $U$ is a nonempty
quasiopen set.}

\medskip

In the next section, we will start developing the basic theory of fine superminimizers.
For this purpose, we first need to define appropriate fine Sobolev spaces.
Here \p-strict subsets will play a key role, as a substitute for relatively compact subsets.

Recall that $V\Subset U$ if $\overline{V}$ is a compact subset of $U$.

\begin{deff}\label{def-fineloc}
  A set $A \subset E$  is a \emph{\p-strict subset} of $E$ if
there is a function $\eta \in \Np_0(E)$ such that $\eta =1$ on
$A$.

  A function $u$ belongs to $\Npploc(U)$ if $u \in \Np(V)$
for all finely open \p-strict subsets $V \Subset U$.
\end{deff}

Equivalently, in the definition of \p-strict subsets
  it can in addition be required
that $0 \le \eta\le 1$,
as in Kilpel\"ainen--Mal\'y~\cite{KiMa92}.
Note that $A \subset E$ is a \p-strict subset of $E$ if and
  only if $\cp(A,E)<\infty$.
  If $V$ is finely open, then
by Lemma~3.3 in Bj\"orn--Bj\"orn--Latvala~\cite{BBLat3},
$V$ has a base of fine neighbourhoods consisting only of
  \p-strict subsets of $V$.
  We recall that functions in $\Npploc(U)$ are
finite q.e., finely continuous q.e.\ and
quasicontinuous, by Theorem~4.4 in~\cite{BBLat3}.

Throughout the paper, we consider minimal
\p-weak upper gradients in $U$.
The following fact is then convenient: If $u\in \Nploc(X)$ then
the minimal \p-weak upper gradients $g_{u,U}$ and $g_u$ with respect to
$U$ and $X$, respectively, coincide a.e.\ in $U$, see
Bj\"orn--Bj\"orn~\cite[Corollary~3.7]{BBnonopen} or
\cite[Lemma~4.3]{BBLat3}.
For this reason we drop $U$ from the notation and simply write $g_u$.

\begin{remark} \label{rmk-Npfineloc-def}
By Proposition~1.48 in~\cite{BBbook},
\p-weak upper gradients do not see
sets of capacity zero.
Thus it follows from  Theorem~\ref{thm-quasiopen-char}
that $u \in \Npploc(U)$
if and only if
$u \in \Np(V)$ for every quasiopen \p-strict subset
$V \Subset U$.
\end{remark}

Recall that the local space $\Nploc(U)$ was introduced in
Section~\ref{sect-prelim}. It is not hard to see that $\Nploc(U) \subset \Npploc(U)$.
However, we do not
 know if the equality $\Nploc(U)=\Npploc(U)$ holds for all quasiopen sets $U$.
 For the reader's convenience, we recall
in Lemma~\ref{lem-open-newton-equal} why these spaces coincide for open sets.

\begin{remark} The reason for the compact inclusion $V \Subset U$ in Definition~\ref{def-fineloc}
is that if it was replaced by
$V \subset U$, the inclusion $\Nploc(U) \subset \Npploc(U)$ might fail.
Neither would Lemma~\ref{lem-open-newton-equal} and Corollary~\ref{cor-fine-min} below hold.

To see this let $U=B(0,2) \setm \{0\} \subset \R^n$, with $1 < p < n$,
in which case it is easy to see that $V=B(0,1) \setm \{0\}$ is an open \p-strict
subset of $U$, but $V \not\Subset U$.
At the same time $u(x)=|x|^{(p-n)/(p-1)} \in \Nploc(U)$ but $u \notin \Np(V)$.
Since $u$ is \p-harmonic in $U$ also Corollary~\ref{cor-fine-min} would fail.
\end{remark}

\begin{lem}\label{lem-open-newton-equal}
For open $G\subset X$ we have $\Nploc(G)=\Npploc(G)$.
\end{lem}

\begin{proof}
First, let $f \in \Nploc(G)$ and let $V \Subset G$ be a
finely open \p-strict subset. For $x \in \overline{V}$ there is $r_x>0$ such that $f \in \Np(B(x,r_x))$.
Since $\overline{V}$ is compact there is a finite subcover
$\{B(x_j,r_{x_j})\}_{j=1}^m$ such that
$\overline{V} \subset \bigcup_{j=1}^m B(x_j,r_{x_j})$.
It follows that $\|f\|_{\Np(\overline{V})}^p \le \sum_{j=1}^m
      \|f\|_{\Np(B(x_j,r_{x_j}))}^p < \infty$, and thus
     $f \in \Np(V)$. Hence $f\in\Npploc(G)$.

Conversely, assume that that $f \in \Npploc(G)$ and $x \in G$.
Then there is $r_x$ such that $B(x,r_x) \Subset G$.
It is straightforward to see that $B(x,r_x)$ is a \p-strict
subset of $G$, and thus
$f \in \Np(B(x,r_x))$.
Hence $f \in \Nploc(G)$.
\end{proof}

The following density result will play a crucial role.

\begin{prop} \label{prop-density}
Let $E \subset X$ be an arbitrary set  and $0\le u \in \Np_0(E)$.
Then there exist  finely open \p-strict subsets
$V_j \Subset E$
and bounded functions
$u_j \in \Np_0(V_j)$ such that
\begin{enumerate}
  \item
$V_j \subset V_{j+1}$ and $0 \le u_j \le u_{j+1} \le u$
    for $j=1,2,\ldots$\,\textup{;}
\item
  $\|u-u_j\|_{\Np(X)} \to 0$ and $u_j(x) \to u(x)$
  for q.e.\ $x \in X$, as $j \to \infty$.
\end{enumerate}
We may also require that $u_j \equiv 0$ outside $V_j$.
\end{prop}

\begin{proof}
Let $U=\fineint E$.
By Theorem~7.3 in Bj\"orn--Bj\"orn~\cite{BBnonopen}, $u \in \Np_0(U)$
and $u=0$ q.e.\ in $X \setm U$.
In the rest of the proof we therefore replace $E$ by $U$,
which is quasiopen by Theorem~\ref{thm-quasiopen-char}.
Modifying $u$ in a set of zero capacity, we can also assume that
  $u\equiv 0$ in $X\setm U$.

By truncating and multiplying by a constant and by a cutoff function,
we may assume that $0 \le u \le 1$ and that $u$ has bounded support,
see the proof of Lemma~5.43 in~\cite{BBbook}.
As $U$ is quasiopen
and $u$ is quasicontinuous on $X$
(by Theorem~\ref{thm-newton-quasicont}),
there are open sets
$G_j$ such that $\Cp(G_j) <2^{-jp}$, $U \cup G_j$ is open and
$u|_{X \setm G_j}$ is continuous, $j=1,2,\ldots$\,.
We can then also find $\psi_j \in \Np(X)$
such that $0 \le \psi_j \le 1$, $\|\psi_j\|_{\Np(X)} < 2^{-j}$
and $\psi_j=1$ in $G_j$.
Let $\phi_k=\min\bigl\{1,\sum_{j=k}^\infty \psi_j\bigr\}$.
Then $\|\phi_k\|_{\Np(X)} <2^{1-k}$ and the sequence $\{\phi_k\}_{k=1}^\infty$
is decreasing. By dominated convergence, $\phi_k(x) \to 0$ for a.e.~$x$.

Next, let
\[
  v_j=(1-\phi_j)u,
 \quad
u_j = (v_j- 2^{-j})_\limplus
  \quad \text{and} \quad
  W_j =\{x \in X : u_j(x) >0\}.
\]
As $\phi_j$ and $u$ are bounded it follows from
the Leibniz rule \cite[Theorem~2.15]{BBbook} that
$v_j \in \Np(X)$, and thus also $u_j \in \Np(X)$.
Hence, by Theorem~\ref{thm-quasiopen-char},
$W_j$ is quasiopen and there is
a set $E_j$ with zero capacity such
that $W_j \setm E_j$ is finely open.
Let $V_j=W_j \setm \bigcup_{i=1}^\infty E_i$.
Then $u_j \in \Np_0(V_j)$
and $u_j \le u$.

By the continuity of $u|_{X \setm G_j}$ and
since $u_j=0$ in the open set $G_j$, we see that
\[
\overline{V}_j\Subset
 \supp u_j=\overline{W}_j \subset \{x : u(x) \ge 2^{-j}\} \setm G_j \subset U.
\]
Note that $\supp u_j$ is bounded since $\supp u$ is bounded.
As $\{\phi_k\}_{k=1}^\infty$ is decreasing, $\{v_j\}_{j=1}^\infty$
is increasing.
Since $v_{j+1} \ge v_j > 2^{-j}$ in $W_j$, we see that
$u_{j+1} \ge 2^{-{j-1}}$ in $W_j \supset V_j$,
from which we conclude that $V_j$ is a \p-strict subset
of $U$ as well as of~$E$.

We next want to show that
\[
   u-u_j=(u-v_j) + (v_j-u_j)
   \to 0
   \quad \text{in } \Np(X).
\]
First,  $\|u-v_j\|_{L^p(X)} \to 0$
and $\|v_j-u_j\|_{L^p(X)}^p \le 2^{-jp} \mu(\supp u) \to 0$.
Next, we see that (using the Leibniz rule \cite[Theorem~2.15]{BBbook})
\begin{equation} \label{eq-u-v_j}
   g_{u-v_j} \le u g_{\phi_j} + \phi_j g_u
    \le  g_{\phi_j} + \phi_j g_u.
\end{equation}
Since $g_{\phi_j} \to 0$ in $L^p(X)$, $\phi_j \to 0$ a.e.,
and $g_u \in L^p(X)$,
the right-hand side in \eqref{eq-u-v_j} tends to $0$ in $L^p(X)$,
by dominated convergence.
Also
\[
   g_{v_j-u_j}
   \le ( g_u +  g_{\phi_j}) \chi_{\{0 < v_j < 2^{-j}\}}
   \le  g_u\chi_{\{0 < v_j < 2^{-j}\}} +  g_{\phi_j}
   \to 0 \quad \text{in } L^p(X),
\]
by dominated convergence since $\chi_{\{0 < v_j < 2^{-j}\}}(x) \to 0$
for a.e.\ $x$.
We thus conclude that $\|u-u_j\|_{\Np(X)} \to 0$ as $j \to \infty$.

By construction, $V_j \subset V_{j+1}$ and  $0 \le u_j \le u_{j+1} \le u$
for $j=1,2,\ldots$\,.
It then follows from Corollary~1.72 in~\cite{BBbook},
that $u_j(x) \to u(x)$ for q.e.\ $x \in
X$, as $j \to \infty$. After replacing $u_j$ by $u_j \chi_{V_j}$ one can
also require that $u_j \equiv 0$ on $X \setm V_j$.
\end{proof}

\section{Fine (super)minimizers}
\label{sect-fine-min}

\begin{deff}\label{def-finesuper}
A function $u \in \Npploc(U)$ is a
\emph{fine minimizer\/ \textup{(}resp.\ fine
superminimizer\/\textup{)}} in $U$ if
\begin{equation}   \label{eq-def-finesuper}
\int_{V} g_{u}^p \, d\mu
\le  \int_{V} g_{u+\phi}^p \, d\mu
\end{equation}
for every
finely open
\p-strict subset $V \Subset U$ and for every
(resp.\ every nonnegative) $\phi \in \Np_0(V)$.

Moreover, $u$ is a \emph{fine subminimizer}
if $-u$ is a fine superminimizer.
\end{deff}

By Remark~\ref{rmk-Npfineloc-def}, we may equivalently consider quasiopen \p-strict subsets $V \Subset U$ in Definition~\ref{def-finesuper}.

\begin{remark} \label{rmk-compare-finemin-U-fineint}
It follows from Proposition~\ref{prop:beyond} below
that if $u \in \Npploc(U)$ then $u$ is a
fine (super)minimizer in $U$ if and only if it is a fine
(super)minimizer in $\fineint U$.
On the other hand, this equivalence is not true if we drop
the assumption $u \in \Npploc(U)$ as seen
in Example~\ref{ex-open-plus-pt} below.
\end{remark}

For the reader's convenience, let us first
look at  the Euclidean case considered in Kilpel\"ainen--Mal\'y~\cite{KiMa92}.
By Remark~\ref{rmk-Npfineloc-def} and~\cite[Theorem~1.1]{BBLat3}
the spaces $\Np(U)$, $\Npploc(U)$ and $\Np_0(U)$
are equal (up to a.e.-equivalence) to
the spaces $W^{1,p}(U)$, $\Wploc(U)$ and $W^{1,p}_0(U)$ defined
for quasiopen subsets of (unweighted) $\R^n$ in~\cite{KiMa92}.
See also Theorem~\ref{thm-Np0} below
and~\cite[Theorem~2.10]{KiMa92}.
This is in particular true for open $U$, in which case
$\Np(U)$ also agrees with the Sobolev space $H^{1,p}(U)$ in
Heinonen--Kilpel\"ainen--Martio~\cite{HeKiMa}
(up to refined equivalence classes) also on weighted $\R^n$.

We next show that the fine supersolutions of~\cite{KiMa92} coincide with our fine superminimizers in $\R^n$.
Recall that, for any $v\in \Npploc(U)$, with $U \subset \R^n$ quasiopen,
we have
\begin{equation}   \label{eq-nabla=g}
|\nabla v|=g_v \quad \text{a.e.\ in }U,
\end{equation}
where $\nabla v$ is as defined in~\cite{KiMa92}; see \cite[Theorem~5.7]{BBLat3}.
The proof and the details above apply equally well if $\R^n$ is equipped
with a \p-admissible measure.

\begin{prop}\label{prop-fine-supersolution}
Let $U\subset \R^n$ be quasiopen and let $u\in \Npploc(U)$.
Then $u$ is a fine superminimizer in $U$ if and only if $u$ is a fine supersolution of
\begin{equation}\label{eq-fine-plaplace}
-\diverg(|\nabla u|^{p-2}\nabla u)=0
\end{equation}
in $U$ in the sense of
Kilpel\"ainen--Mal\'y\/~\textup{\cite[Section~3.1]{KiMa92}},
i.e.\
\begin{equation}      \label{eq-fine-supersol}
          \int_{V} |\nabla u|^{p-2} \nabla u \cdot \nabla \phi\,dx \ge 0
\end{equation}
for all \p-strict subsets $V \Subset U$ and all bounded nonnegative
$\phi\in\Np_0(V)$.
\end{prop}

\begin{proof}
First, let $u$ be a fine supersolution of \eqref{eq-fine-plaplace} in
$U$ and
let $V\Subset U$ be a \p-strict subset of $U$. Let $\varphi\in
\Np_0(V)$, $\varphi\ge 0$.
Assuming also that $\varphi$ is bounded, we
obtain from \eqref{eq-fine-supersol} that
\begin{align*}
\int_{V}|\nabla u|^p\,dx&=\int_{V}|\nabla u|^{p-2}\nabla u\cdot \nabla u\,dx
\le \int_{V}|\nabla u|^{p-2}\nabla u\cdot \nabla (u+\varphi)\,dx\\
&\le \biggl(\int_{V}|\nabla u|^p\,dx\biggr)^{1-1/p}
   \biggl(\int_{V}|\nabla (u+\varphi)|^p\,dx\biggr)^{1/p}.
\end{align*}
Since $u \in \Np(V)$, the first integral on the right-hand side is finite,
and dividing by it shows that
\[
 \int_{V}|\nabla u|^p\,dx \le \int_{V}|\nabla (u+\varphi)|^p\,dx.
\]
If $\varphi$ is not bounded, then dominated convergence implies that
\[
\int_{V} |\nabla (u+\phi)|^p \, dx
   = \lim_{k \to \infty} \int_{V} |\nabla (u+\min\{\phi,k\})|^p \, dx.
\]
Using also \eqref{eq-nabla=g} shows that  $u$ is a fine superminimizer in the sense of Definition~\ref{def-finesuper}.

For the converse implication, assume that $u$ is a fine superminimizer
in $U$. Let $V\Subset U$ be a \p-strict subset of $U$ and let
$\varphi\in \Np_0(V)$ be bounded and nonnegative.
Using \eqref{eq-nabla=g}, we have for any $0<\eps<1$ that
\[
\int_V|\nabla u|^p\,dx\le\int_V|\nabla (u+\eps\varphi)|^p\,dx,
\]
and therefore
\[
\int_V\frac{|\nabla (u+\eps\varphi)|^p-|\nabla u|^p}{\eps}\,dx\ge 0.
\]
From this the inequality
\[
\int_U|\nabla u|^{p-2}\nabla u\cdot \nabla \varphi\,dx\ge 0
\]
follows in the same way as in the proof of
Theorem~5.13 in Heinonen--Kilpel\"ainen--Martio~\cite{HeKiMa}.
\end{proof}

\begin{lem}\label{lem-fine-min-2}
  A function $u$ is a fine minimizer in $U$ if and only
  if it is both a fine subminimizer and a fine superminimizer in $U$.
\end{lem}

\begin{proof} Assume that $u$ is
both a fine subminimizer and a fine superminimizer in $U$.
Let $V\Subset U$ be a finely open \p-strict subset and let $\phi \in \Np_0(V)$.
We may assume that $\phi=0$ everywhere in $X\setminus V$.
Since $\{\phi_\limpm \ne 0\}$  are quasiopen \p-strict subsets of
$U$ (by Theorem~\ref{thm-quasiopen-char}), testing
\eqref{eq-def-finesuper} with $\phi_\limpm$ implies that
\begin{align*}
    \int_{\{\phi \ne 0\}} g_{u}^p \, d\mu
    &= \int_{\{\phi_\limplus \ne 0\}} g_{u}^p \, d\mu
       + \int_{\{\phi_\limminus \ne 0\}} g_{u}^p \, d\mu \\
  & \le \int_{\{\phi_\limplus \ne 0\}} g_{u+ \phi_\limplus}^p \, d\mu
             + \int_{\{\phi_\limminus \ne 0\}} g_{u-\phi_\limminus}^p \, d\mu
    =  \int_{\{\phi \ne 0\}} g_{u+\phi}^p \, d\mu,
\end{align*}
see Remark~\ref{rmk-Npfineloc-def}. Adding $\int_{V\cap\{\phi=0\}}g_{u}^p \, d\mu=\int_{V\cap\{\phi=0\}}g_{u+\phi}^p \, d\mu$
to both sides shows that $u$ is a fine minimizer. The converse implication is trivial.
\end{proof}

The following characterization
is quite convenient.
It also shows that
condition \eqref{eq-def-finesuper} in Definition~\ref{def-finesuper}
can equivalently
be
required to hold for
arbitrary $V \subset U$.

\begin{lem} \label{lem-char-1}
  Let   $u \in \Npploc(U)$.
  Then $u$ is a fine\/ \textup{(}super\/\textup{)}minimizer
in $U$
if and only if
\begin{equation} \label{eq-fine-supermin-char}
\int_{\{\phi \ne 0\}} g_{u}^p \, d\mu
\le  \int_{\{\phi \ne 0\}} g_{u+\phi}^p \, d\mu
\end{equation}
for every\/ \textup{(}nonnegative\/\textup{)} $\phi \in \Np_0(U)$.
\end{lem}

Note that for some $\phi$ the integrals
in \eqref{eq-fine-supermin-char}
may be infinite,
but then they are always infinite
simultaneously. The characterization in Lemma~\ref{lem-char-1}
is in contrast to the definition~\eqref{eq-fine-supersol}
of supersolutions, where $V=U$ is allowed only if $u\in \Np(U)$.

\begin{proof}
Assume first that $u$ is a fine superminimizer and
that $\phi \in \Np_0(U)$ is nonnegative.
By Proposition~\ref{prop-density}, there are
finely open \p-strict subsets $V_j \Subset U$
and functions $\phi_j \in \Np_0(V_j)$
such that $0 \le \phi_j \le \phi$ and
\begin{equation} \label{eq-phij}
  \lim_{j \to \infty} \|\phi_j - \phi\|_{\Np(X)} =0.
\end{equation}
Since  $u$ is a fine superminimizer, we see that
\[
   \int_{V_j} g_{u}^p \, d\mu
  \le  \int_{V_j} g_{u+\phi_j}^p \, d\mu
  =
  \int_{V_j \cap \{\phi \ne 0\}} g_{u+\phi_j}^p \, d\mu
  + \int_{V_j \cap \{\phi = 0\}} g_{u}^p \, d\mu.
\]
As $u \in \Np(V_j)$ the last term is finite, and we can thus subtract it
from both sides in the inequality obtaining
\begin{align*}
\int_{\{\phi \ne 0\}} g_{u}^p \, d\mu
  &= \int_{\{\phi \ne 0\} \setm V_j} g_{u+\phi_j}^p \, d\mu
     + \int_{V_j \cap \{\phi \ne 0\}} g_{u}^p \, d\mu \\
  &\le  \int_{\{\phi \ne 0\} \setm V_j} g_{u+\phi_j}^p \, d\mu
     + \int_{V_j \cap \{\phi \ne 0\}} g_{u+\phi_j}^p \, d\mu
   =  \int_{\{\phi \ne 0\}} g_{u+\phi_j}^p \, d\mu,
\end{align*}
which together with \eqref{eq-phij}  shows
that \eqref{eq-fine-supermin-char} holds.

Conversely, let $V \Subset U$ be a finely open
\p-strict subset  and $\phi \in \Np_0(V)$ be nonnegative.
It then follows from
\eqref{eq-fine-supermin-char}
and the fact that $g_u=g_{u+\phi}$ on $\{x:\phi(x)=0\}$,
  that \eqref{eq-def-finesuper}
holds and thus $u$ is a fine superminimizer.
The claim for fine minimizers follows from
Lemma~\ref{lem-fine-min-2}.
\end{proof}

\begin{cor} \label{cor-fine-min}
  Let $G$ be an open set.
  Then $u$ is a fine\/ \textup{(}super\/\textup{)}minimizer in $G$
  if and only if it is a\/ \textup{(}super\/\textup{)}minimizer in $G$.
\end{cor}

Here we define (super)minimizers as in
Definition~7.7 in~\cite{BBbook}.

\begin{proof}
  Since $G$ is open, $\Nploc(G)=\Npploc(G)$
by Lemma~\ref{lem-open-newton-equal}.
The equivalence then follows directly from Lemma~\ref{lem-char-1}
together with \cite[Proposition~7.9]{BBbook}.
\end{proof}

\begin{lem} \label{lem-paste-supermin}
\textup{(Pasting lemma)}
Assume that\/ $U_1 \subset U_2$  are finely open sets, and
that $u_1$  and $u_2$ are fine superminimizers in\/ $U_1$ and\/ $U_2$,
respectively.
Let
\[
     u=\begin{cases}
        u_2 & \text{in\/ } U_2 \setm U_1, \\
        \min\{u_1,u_2\} & \text{in\/ } U_1.
	\end{cases}
\]
If $u \in \Npploc(U_2)$, then $u$ is a fine superminimizer in\/ $U_2$.
\end{lem}

\begin{proof} Assume that $u \in \Npploc(U_2)$.
Let $V\Subset U_2$ be a finely open \p-strict subset and $0\le \phi \in \Np_0(V)$.

Let $\phi_2=(u+\phi-u_2)_\limplus$ and $\phi_1=\phi-\phi_2$.
Lemma~2.37 in~\cite{BBbook} implies that $\phi_2\in\Np_0(V)$.
It is also easily verified that $\phi_1=0$ outside $U_1$  and hence
$\phi_1\in\Np_0(U_1)$.
Lemma~\ref{lem-char-1} applied to $u_1$, together with the facts
that $u=u_1$ when $\phi_1>0$ and
$u=u_2$ when $\phi=\phi_2>0$, therefore yields
\begin{align*}
\int_{\{\phi >0\}} g^p_u \,d\mu &=  \int_{\{\phi_1>0\}} g^p_{u_1} \,d\mu
                 + \int_{\{\phi=\phi_2>0\}} g^p_{u_2} \,d\mu \\
&\le \int_{\{\phi_1> 0\}} g^p_{u_1+\phi_1} \,d\mu
                 + \int_{\{\phi=\phi_2>0\}} g^p_{u_2} \,d\mu \\
&= \int_{\{\phi=\phi_1> 0\}} g^p_{u+\phi} \,d\mu
              + \int_{\{\phi>\phi_1>0\}} g^p_{u+\phi_1} \,d\mu
              + \int_{\{\phi=\phi_2>0\}} g^p_{u_2} \,d\mu.
\end{align*}
Note that $u+\phi_1=u_2$ in
$\{x:\phi(x)>\phi_1(x)>0\}=\{x:\phi(x)>\phi_2(x)>0\}$.
Summing the last two integrals, we thus obtain
\begin{align*}
\int_{\{\phi >0\}} g^p_u \,d\mu &\le \int_{\{\phi=\phi_1>0\}} g^p_{u+\phi} \,d\mu
              + \int_{\{\phi_2>0\}} g^p_{u_2} \,d\mu \\
&\le \int_{\{\phi=\phi_1>0\}} g^p_{u+\phi} \,d\mu
              + \int_{\{\phi_2> 0\}} g^p_{u_2+\phi_2} \,d\mu,
\end{align*}
where in the last step we used Lemma~\ref{lem-char-1}, applied to $u_2$.
Since $u_2+\phi_2=u+\phi$ in $\{\phi_2>0\}$, we conclude that
\[
\int_V g^p_u \,d\mu = \int_{\{\phi >0\}} g^p_u \,d\mu
    + \int_{V\cap\{\phi=0\}} g^p_{u+\phi} \, d\mu
  \le \int_V g^p_{u+\phi} \, d\mu,
\]
i.e.\ \eqref{eq-def-finesuper} in the definition of fine superminimizers
holds.
\end{proof}

\begin{cor} \label{cor-min-supermin}
If $u$ and $v$ are fine superminimizers in $U$,
then $\min\{u,v\}$ is also a fine superminimizer in $U$.
\end{cor}

Assume that $E$ is an arbitrary measurable set. Then the space
$\Npploc(E)$ as well as fine minimizers and fine superminimizers in
$E$ can be defined in the same way as in Definitions~\ref{def-fineloc}
and~\ref{def-finesuper} (just replacing $U$ be $E$). The following
characterization suggests that the notions of fine superminimizers and
minimizers might not be very interesting beyond quasiopen sets.

\begin{prop}     \label{prop:beyond}
Let $E$ be measurable and assume that $u \in \Npploc(E)$.
Then $u$ is a  fine\/ \textup{(}super\/\textup{)}\-mi\-ni\-mi\-zer in $E$
if and only if it is a  fine\/ \textup{(}super\/\textup{)}minimizer in
$V:=\fineint E$.
\end{prop}

\begin{proof}
Assume that $u$ is a fine superminimizer in $V$, and let $\phi \in \Np_0(E)$ be nonnegative.
By Theorem~7.3 in Bj\"orn--Bj\"orn~\cite{BBnonopen} we see that $\phi \in \Np_0(V)$.
By Lemma~\ref{lem-char-1},
\[
   \int_{\{\phi \ne 0\}} g_{u}^p \, d\mu
\le  \int_{\{\phi \ne 0\}} g_{u+\phi}^p \, d\mu.
\]
Since Proposition~\ref{prop-density} holds for $E$, so does
Lemma~\ref{lem-char-1}, from which it follows that $u$ is a fine superminimizer in $E$.
The converse implication is clear and the proof for fine minimizers is similar.
\end{proof}

\section{The obstacle and Dirichlet problems}
\label{sect-obst}

The obstacle problem will be a fundamental tool for studying fine minimizers.

\begin{deff} \label{deff-obst-E}
Assume that $U$ is bounded and $\Cp(X \setm U)>0$.
Let $f \in \Np(U)$ and $\psi : U \to \eR$.
Then we define
\begin{equation*}
    \K_{\psi,f}(U)=\{v \in \Np(U) : v-f \in \Np_0(U)
            \text{ and } v \ge \psi \ \text{q.e. in } U\}.
\end{equation*}
A function $u \in \K_{\psi,f}(U)$
is a \emph{solution of the $\K_{\psi,f}(U)$-obstacle problem}
if
\[
       \int_U g^p_{u} \, d\mu
       \le \int_U g^p_{v} \, d\mu
       \quad \text{for all } v \in \K_{\psi,f}(U).
\]
\end{deff}

The \emph{Dirichlet problem} is a special case of the obstacle
problem, with the trivial obstacle $\psi \equiv -\infty$.
Note that the boundary data $f$ are only required to belong to $\Np(U)$,
i.e.\ $f$ need not be defined on $\bdry U$
or the fine boundary $\bdyp U$.

\begin{thm}\label{thm-obstacle}
Assume that $U$ is bounded and $\Cp(X \setm U)>0$.
Let $f \in \Np(U)$  and $\psi : U \to \eR$,
and assume that $\K_{\psi,f}(U) \ne \emptyset$.
Then there exists a solution $u$ of the $\K_{\psi,f}(U)$-obstacle problem,
and this solution is unique q.e.
Moreover,  $u$ is a fine superminimizer in $U$.

If $\psi \equiv -\infty$ in $U$ or if $\psi$ is a fine subminimizer
in $U$, then $u$ is a fine minimizer in $U$.
\end{thm}

\begin{proof}
The existence and q.e.-uniqueness follow from
Theorem~4.2 in Bj\"orn--Bj\"orn~\cite{BBnonopen}.

To show that $u$ is a fine (super)minimizer in $U$,
let $V\Subset U$ be a finely open \p-strict subset and let $\phi \in \Np_0(V)$.
If
$\psi$ is not a fine subminimizer and $\psi \not \equiv -\infty$,
then we also require $\phi$ to be nonnegative.

It is easily verified that $v:=\max\{u+\phi,\psi\} \in \K_{\psi,f}(U)$.
Hence, as $u$ is a solution of the $\K_{\psi,f}(U)$-obstacle problem,
we get that
\begin{align} \label{eq-lem-obstacle}
\int_{U} g_{u}^p \, d\mu
&\le  \int_{U} g_{v}^p \, d\mu
= \int_{\{u+\phi\ge \psi\}} g_{u+\phi}^p \, d\mu
       + \int_{\{u+\phi<\psi\}} g_{\psi}^p \, d\mu \\
&\le \int_{\{u+\phi\ge \psi\}} g_{u+\phi}^p \, d\mu
       + \int_{\{u+\phi<\psi\}} g_{u+\phi}^p \, d\mu
= \int_{V} g_{u+\phi}^p \, d\mu
       + \int_{U\setm V} g_{u}^p \, d\mu, \nonumber
\end{align}
where the second inequality is justified by Lemma~\ref{lem-char-1}
if $\psi$ is a fine subminimizer, and is trivial otherwise as $u+\phi\ge\psi$
q.e.\ in $U$ in that case.

Since $u \in\Np(U)$, we see that the last integral in \eqref{eq-lem-obstacle}
is finite and
subtracting it from both sides of \eqref{eq-lem-obstacle} yields \eqref{eq-def-finesuper}
in Definition~\ref{def-finesuper} for the above choices of $V$ and
$\phi\in\Np_0(V)$.
As $V$ was arbitrary, it follows that $u$ is a fine superminimizer in $U$.
When $\phi$ is not required to be nonnegative, we
conclude that $u$ is a fine minimizer in~$U$.
\end{proof}

Note that there is a comparison principle
for solutions of obstacle problems, see Corollary~4.3 in
Bj\"orn--Bj\"orn~\cite{BBnonopen}.

\section{Fine continuity for solutions of the Dirichlet problem}
\label{sect-fine-cont}

\emph{In this section we assume that $U$ is
a nonempty finely open  set.
Except for Theorem~\ref{thm-Np0},
we also assume that $U$ is bounded and that $\Cp(X \setm U)>0$.}

\medskip

We do not know in general if fine minimizers have finely continuous
representatives.
However in this section we obtain sufficient conditions for the fine
continuity of solutions of the (fine) Dirichlet
problem, and deduce Theorem~\ref{thm-finecont-dir-intro}.
The proof of our key Lemma~\ref{lem-fine-cont-supermin}
below was inspired by the proof of Theorem~5.3
in Kilpel\"ainen--Mal\'y~\cite{KiMa92}.
As we study fine continuity in this section it
is natural to consider only finely open sets $U$.

With continuous boundary data, the solution of the Dirichlet
problem in an open set need not be continuous
at an irregular boundary point. However, the solution is finely continuous.
We demonstrate this by the following
example using Corollary~\ref{cor-fine-lim} below.

\begin{example}  \label{ex-only-fine-cont}
Consider, for example, a bounded open set $G \subset X$
with $\Cp(X \setm G)>0$ and a strongly irregular boundary point
$z \in \bdy G$, see Bj\"orn~\cite[p.~40]{ABclass} (or \cite[Definition~13.1]{BBbook}).

Then $X \setm G$ is thin at $z$, by
the sufficiency part of the Wiener criterion,
see Bj\"orn--MacManus--Shanmugalingam~\cite[Theorem~5.1]{BMS} and
J.~Bj\"orn~\cite[p.~370 and Corollary~3.11]{JB-pfine}
(or \cite[Theorem~11.24]{BBbook}).
Thus $U=G \cup \{z\}$ is finely open.
Moreover  $\Cp(\{z\})=0$, by the Kellogg property, see
  Bj\"orn--Bj\"orn--Shan\-mu\-ga\-lin\-gam~\cite[Theorem~3.9]{BBS}
  (or \cite[Theorem~10.5]{BBbook}).
Hence $\Cp(X \setm U)>0$.

Since $z$ is strongly irregular, it follows from Theorem~13.13 in~\cite{BBbook} that
the continuous solution $h$ of the
$\K_{-\infty,d}(G)$-obstacle problem, with $d(x)=d(x,z)$,
does not have a limit at $z$.
However, by Corollary~\ref{cor-fine-lim} below, $h$ does
have a fine limit.
\end{example}

We will need the following  auxiliary result,
which may also be of independent interest.
In what follows, the notions of $\finelim$, $\finelimsup$ and
$\fineliminf$ are defined using punctured fine neighbourhoods.
Note that since
\[
\cp(B(x,r) \setm\{x\},B(x,2r))= \cp(B(x,r),B(x,2r)),
\]
there are no isolated points in the fine topology,
i.e.\ no singleton sets are finely open.

\begin{thm}   \label{thm-Np0}
Let $U\subset V \subset X$ be finely open sets.
Assume that $u \in \Np(U)$ and extend it by $0$ to $V \setm U$.
Then the following are equivalent\/\textup{:}
\begin{enumerate}
\item \label{k-Np0}
$u \in \Np_0(U,V)$, i.e.\ $u\in\Np(V)$\textup{;}
\item \label{k-u-qcont}
$u$ is quasicontinuous in $V$\textup{;}
\item \label{k-finecont}
$u$ is finite q.e.\ and finely continuous q.e.\ in $V$\textup{;}
\item \label{k-curve}
$u$ is measurable, finite q.e.,
and $u\circ \ga$ is continuous for \p-almost every
curve $\ga : [0, l_\ga] \to V$\textup{;}
\item \label{k-f-lim=0}
$\displaystyle  \finelim_{U \ni y\to x} u(y)=0$
 for q.e.\ $x\in V \cap \bdyp U$.
\end{enumerate}
\end{thm}

We will only need the equivalence \ref{k-Np0} $\eqv$ \ref{k-f-lim=0}
(when proving Lemma~\ref{lem-fine-cont-supermin}).
However, when deducing this equivalence we will rely on several earlier results,
which essentially requires us to obtain the full equivalence of
\ref{k-Np0}--\ref{k-f-lim=0}.

\begin{proof}
  \ref{k-Np0} $\imp$ \ref{k-u-qcont} $\eqv$ \ref{k-finecont}
  These implications hold by
    Theorem~\ref{thm-newton-quasicont}.

\ref{k-u-qcont} $\eqv$ \ref{k-curve}
This follows from Theorem~1.2 in Bj\"orn--Bj\"orn--Mal\'y~\cite{BBMaly}.

\ref{k-curve} $\imp$ \ref{k-Np0}
Let $g\in L^p(U)$
be a \p-weak upper gradient of $u$ in $U$, extended by zero to $V\setm U$.
Consider a curve~$\ga$ as in \ref{k-curve} such that
none of its subcurves in $U$ is exceptional in \eqref{ug-cond} for the pair $(u,g)$.
Lemma~1.34\,(c) in \cite{BBbook} implies that \p-almost every curve
  has this property.
If $\ga\subset U$ or $\ga\subset V\setminus U$, there is nothing to prove.
Hence by splitting $\ga$ into two parts, if necessary, and possibly
reversing the direction, we may assume that $x=\ga(0)\in U$ and $y=\ga(l_\ga)\notin U$.
Let
\(
c= \inf\{t: \ga(t)\notin U\}
\)
and $y_0=\ga(c)$. By continuity, $u(y_0)=0$, and hence
\[
|u(x)-u(y)| = |u(x)-u(y_0)| = \lim_{\eps\to 0} |u(x)-u\circ\ga(c-\eps)|
\le \int_{\ga|_{[0,c]}} g\,ds \le \int_{\ga} g\,ds.
\]
It follows that $g$ is a \p-weak upper gradient of  $u$ in $V$ and hence $u\in \Np(V)$.

\ref{k-finecont} $\imp$ \ref{k-f-lim=0}
  As $u$ is finely continuous q.e.\ and $u \equiv 0$ in $V \setm U$,
  \ref{k-f-lim=0} follows directly.

\ref{k-f-lim=0} $\imp$ \ref{k-finecont}
Since $u\in\Np(U)$, it is
finely continuous q.e.\ and finite q.e.\ in $U$,
by Theorem~\ref{thm-newton-quasicont}.
Thus $u$ is finite q.e.\ in $V$ and finely continuous
q.e.\ in $V\setm \bdyp U$.
As $u \equiv 0$ in $V \setm U$ and \ref{k-f-lim=0} holds,
$u$ is finely continuous q.e.\ in $V\cap\bdyp U$.
\end{proof}

We define for any function $u:U \to \eR$
the \emph{fine lsc-regularization} $u_*:\clUp \to \eR$ of $u$ as
\[
  u_*(x) = \fineliminf_{U\ni y \to x} u(y),\quad  \text{if } x\in \clUp,
\]
and the \emph{fine usc-regularization} $u^*:\clUp \to \eR$ of $u$ as
\[
  u^*(x) = \finelimsup_{U\ni y \to x} u(y),\quad \text{if } x\in \clUp.
\]
In this paper, we will only regularize Newtonian functions.
As these are finely continuous q.e., we have $u=u_*=u^*$
  q.e.\ in $U$.
We say that $u$ is \emph{finely lsc-regularized} if $u=u_*$
in $U$ and \emph{finely usc-regularized} if $u=u^*$ in $U$.
Note that $u_*$ (resp.\ $u^*$) is finely lsc-regularized
(resp.\ finely usc-regularized) in $U$.
Recall also the characterization of $\bdyp U$ in Lemma~\ref{lem-bdyp}.

\begin{lem} \label{lem-fine-cont-supermin}
Let $z \in U$, $B=B(z,r)$, $f \in \Np(U)$ and let $u$ be a fine
superminimizer in $B\cap U$ such that $u-f\in\Np_0(B \cap U,B)$.
Assume that $c\in \R$ is such that
\[
f_* \ge c \quad \text{q.e.\ in } B\cap \bdyp U.
\]
If $u_*(z)< c$, then $u_*$ is finely continuous at $z$.
\end{lem}

\begin{proof} Assume that
\begin{equation} \label{eq-u<c}
u_*(z)=\fineliminf_{U\ni y \to z} u(x) < c.
\end{equation}
We want to apply the pasting Lemma~\ref{lem-paste-supermin}
to the fine superminimizers $c$ and $u$ in the
finely open sets $B$ and $B\cap U$, respectively.
We therefore show that the function
\[
u_c =\begin{cases}
        c & \text{in\/ } B \setm U, \\
        \min\{u,c\} & \text{in\/ } B\cap U
	\end{cases}
\]
belongs to $\Np(B)$.
This will follow
if we can show that
$u_c-c\in\Np_0(B\cap U,B)$,
which we will do using characterization~\ref{k-f-lim=0} in Theorem~\ref{thm-Np0}.
For this purpose, it suffices to show that
\begin{equation} \label{eq-uc}
\finelim_{U\ni y\to x} \min\{u(y),c\} = c \quad \text{for q.e.\ }x\in B\cap\bdyp U.
\end{equation}
Clearly, $\finelimsup_{U\ni y\to x} \min\{u(y),c\} \le c$ everywhere.
The fact that $u-f\in\Np_0(B \cap U,B)$, together with
Theorem~\ref{thm-Np0}, shows
that for q.e.\ $x\in B\cap\bdyp U$,
\[
\fineliminf_{U\ni y\to x} u(y)
\ge \finelim_{U\ni y\to x} (u-f)(y) + \fineliminf_{U\ni y\to x} f(y)
= f_*(x) \ge c.
\]
Hence, \eqref{eq-uc} holds and $u_c \in \Np(B)$.
Therefore, by Lemma~\ref{lem-paste-supermin}, 
$u_c$ is a fine superminimizer in $B$.
As $B$ is open, $u_c$ is a superminimizer in $B$,
by Corollary~\ref{cor-fine-min}.
It follows from Proposition~7.4 in Kinnunen--Martio~\cite{KiMa02}
(or \cite[Proposition~9.4]{BBbook}), that $u_c$ has a
superharmonic representative $v$ such that $v=u_c$ q.e.\ in $B$.
Thus, $v$ is finely continuous in $B$, by
Bj\"orn~\cite[Theorem~4.4]{JB-pfine} or
Korte~\cite[Theorem~4.3]{korte08} (or \cite[Theorem~11.38]{BBbook}).
As $v= u_c$ q.e.\ in $B \cap U$,
we conclude from \eqref{eq-u<c} that
\[
    v(z) = \fineliminf_{x \to z} u_c(x)
     =\fineliminf_{x \to z} u(x) = u_*(z) < c.
\]
Since $v$ is finely continuous at $z$,
there is a fine neighbourhood $V$ of $z$ contained in $B\cap U$
so that $\sup_V v<c$. Hence also
\[
v(z)= \finelimsup_{x \to z} v(x) = \finelimsup_{x \to z} u_c(x)
=\finelimsup_{x \to z} u(x) = u^*(z).
\qedhere
\]
\end{proof}

In what follows, the $\cpliminf$, $\cplimsup$ and $\cplim$ are taken with respect
to the metric topology from $X$ and up to sets of zero capacity
in punctured neighbourhoods.
For instance, for a function $v$ defined in a set $E$,
\begin{align*}
\cpliminf_{E\ni x\to z} v(x)
& :=\lim_{r \to 0} \cpessinfalt_{E \cap (B(z,r) \setm \{z\})} v \\
& := \lim_{r \to 0} \sup \{ k: \Cp(\{ x\in E \cap (B(z,r) \setm \{z\}): v(x)<k \})=0 \}.
\end{align*}
In particular,
\[
 \cpliminf_{E\ni x\to z} v(x)= \infty
 \quad \text{if } \Cp(E \cap (B(z,r) \setm \{z\}))=0 \text{ for some $r>0$}.
\]

\begin{cor}  \label{cor-obst-prob-liminf}
  Let $z \in U$, $f\in\Np(U)$
  and let $u$ be a fine superminimizer in $U$ such that
$u-f\in\Np_0(U)$.
If
\begin{equation} \label{eq-cpessliminf}
  u_*(z) < \cpliminf_{U\ni x\to z} f(x)
  \quad \text{or} \quad
u_*(z) < \cpliminf_{\bdyp U\ni x\to z} f_*(x),
\end{equation}
then $u_*$ is finely continuous at $z$.
\end{cor}

\begin{proof} It follows directly from the definition of $f_*$ that
\[
  \cpliminf_{U\ni x\to z} f(x) \le \cpliminf_{\bdyp U\ni x\to z} f_*(x),
\]
and thus we can without loss of generality assume
the latter inequality in~\eqref{eq-cpessliminf}.
We can then find $c>u_*(z)$ and $B=B(z,r)$
such that
\[
f_*(x) \ge c
\quad \text{for all } x\in B\cap\bdyp U.
\]
Lemma~\ref{lem-fine-cont-supermin} concludes the proof.
\end{proof}

Note that if $f \in \Np(X)$ in Corollary~\ref{cor-obst-prob-liminf},
then $f=f_*=f^*$ q.e.\ on $\clUp$ and thus
\begin{equation}   \label{eq-esslim-bdypU}
  \cpliminf_{\bdyp U\ni x\to z} f_*(x)
  = \cpliminf_{\bdyp U\ni x\to z} f(x)
\end{equation}
in \eqref{eq-cpessliminf}.

\begin{thm}  \label{thm-fine-cont-if-Cp-lim}
Let $f \in \Np(U)$. Then the finely lsc-regularized solution $h_*$
of the $\K_{-\infty,f}(U)$-obstacle problem is finely continuous at each
$z\in U$ which satisfies one of the following conditions\/\textup{:}
\begin{enumerate}
\item \label{it-cplim-ex}
  The limit
$\displaystyle\cplim_{U\ni x\to z} f(x)$ exists.
\item \label{it-cplim-ex-2}
  The equality
  $\displaystyle\cplim_{\bdyp U\ni x\to z} f_*(x)= \cplim_{\bdyp U\ni x\to z} f^*(x)$
  holds.
\item \label{it-Cp=0}
There  exists $r>0$ such that $\Cp(B(z,r)\cap \bdyp U)=0$. 
\end{enumerate}
\end{thm}

\begin{proof}
\ref{it-cplim-ex}
Assume that $h_*$ (and hence also $h^*$) is not finely continuous at $z$,
  i.e.\ that $h_*(z) < h^*(z)$.
Then Corollary~\ref{cor-obst-prob-liminf}, applied to both
  $h_*$ and $-h^*$,
together with the assumption~\ref{it-cplim-ex} shows that

\[
h_*(z) \ge \cplim_{U\ni x\to z} f(x)
=-  \cplim_{U\ni x\to z} (-f)(x)
 \ge h^*(z),
\]
a contradiction. Hence $h_*$ is finely continuous at $z$.

\ref{it-cplim-ex-2}
Assume, as in \ref{it-cplim-ex}, that $h_*$
is not finely continuous at $z$.
This time, Corollary~\ref{cor-obst-prob-liminf}
shows that
\[
  h_*(z) \ge
 \cplim_{\bdyp U\ni x\to z} f_*(x)
  = -\cplim_{\bdyp U\ni x\to z} (-f^*)(x)
  \ge h^*(z),
\]
a contradiction.
Hence $h_*$ is finely continuous at $z$.

\ref{it-Cp=0}
If $h_*(z)=\infty$, then also $h^*(z)=\infty$ and
$h$ is finely continuous at $z$. Otherwise,
\[
h_*(z) < \infty= \cpliminf_{\bdyp U\ni x\to z} f_*(x),
\]
and the conclusion follows from Corollary~\ref{cor-obst-prob-liminf}.
\end{proof}

We can now prove Theorem~\ref{thm-finecont-dir-intro}.

\begin{proof}[Proof of Theorem~\ref{thm-finecont-dir-intro}]
Let $h_*$ be the finely lsc-regularized solution of the Dirichlet problem,
i.e.\ of the $\K_{-\infty,f}(U)$-obstacle problem,
and let $z \in U$.
If $f \in C(U) \cap \Np(U)$,
then condition~\ref{it-cplim-ex} in Theorem~\ref{thm-fine-cont-if-Cp-lim}
holds. Recall that here $f$ is assumed continuous with values in $\eR$.

On the other hand, if $f \in \Np(X)$, then
either condition~\ref{it-Cp=0} in Theorem~\ref{thm-fine-cont-if-Cp-lim}
is fulfilled, or \eqref{eq-esslim-bdypU} and the continuity of $f$ on
$\clUp \cap \bdy U \supset \bdyp U \cup (U\cap \itoverline{\bdyp U})$
yield
\[
\cplim_{\bdyp U\ni x\to z} f_*(x)
= \cplim_{\bdyp U\ni x\to z} f(x)
=
\cplim_{\bdyp U\ni x\to z} f^*(x),
\]
i.e.\ condition~\ref{it-cplim-ex-2} in Theorem~\ref{thm-fine-cont-if-Cp-lim}
holds.

Thus, in all cases, $h_*$ is finely continuous at $z$ by
Theorem~\ref{thm-fine-cont-if-Cp-lim}.
\end{proof}

As an application of Theorem~\ref{thm-fine-cont-if-Cp-lim} we obtain the
following result. Note that $V$ is finely open since it is the intersection
of the finely open set $U$ and the open set $X \setm \{z\}$.
Moreover, $\bdyp V = \bdyp U \cup\{z\}$ as there are no finely isolated
points, see Lemma~\ref{lem-bdyp}.

\begin{thm} \label{thm-fine-lim}
Let $z \in U$ and $V=U \setm \{z\}$. Let $f \in \Np(V)$ and
let $h_{V}$ be a solution of the $\K_{-\infty,f}(V)$-obstacle problem.
Assume that one of the following holds\/\textup{:}
\begin{enumerate}
\item \label{it-Cp(z)>0}
  $\Cp(\{z\})>0$ and
$\displaystyle f(z):=\finelim_{V\ni x\to z} f(x)$ exists,
  which in particular holds if $f \in \Np(U)$.
\item \label{it-f-cont}
  $\Cp(\{z\})=0$ and
  $\displaystyle f(z):=\cplim_{V\ni x\to z} f(x)$ exists.
\end{enumerate}
Then the fine limit
\begin{equation}   \label{eq-ex-finelim-h}
\finelim_{V\ni x\to z} h_{V}(x)
\quad \text{exists.}
\end{equation}
\end{thm}

\begin{proof}
\ref{it-Cp(z)>0} Note first that if $f\in\Np(U)$, then $f$ is finely continuous
q.e.\ in $U$ and thus $f(z)=\finelim_{V\ni x\to z} f(x)$.
Since $h_{V}-f \in\Np_0(V)$, Theorem~\ref{thm-Np0} implies that
\[
\finelim_{V\ni x\to z} (h_{V}-f)(x)=0
\]
and \eqref{eq-ex-finelim-h} follows.

\ref{it-f-cont}
In this case, $f,h_V \in \Np(U)$, by Proposition~1.48 in~\cite{BBbook}
(where $h_V(z)$ is defined arbitrarily).
Moreover, $h_{V}-f \in \Np_0(V) \subset \Np_0(U)$.
Let $h_{U}$ be a
solution of the $\K_{-\infty,f}(U)$-obstacle problem.
Then $h_{U}-f \in \Np_0(U)$ and by
the uniqueness part of Theorem~\ref{thm-obstacle}, we conclude that
$h_{U}=h_{V}$ q.e.\ in $U$.
Theorem~\ref{thm-fine-cont-if-Cp-lim} and the assumption
\ref{it-f-cont} imply that  $h_{U}$ is finely continuous at $z$,
and thus by Theorem~\ref{thm-Np0},
\[
\finelim_{V\ni x\to z}h_{V}(z)=\finelim_{V\ni x\to z} h_{U}(x)=h_{U}(z).
\qedhere
\]
\end{proof}

\begin{cor} \label{cor-fine-lim}
Let $G$ be a nonempty bounded open set with $\Cp(X \setm G)>0$
and let $z \in \bdy G$.
Let $f \in \Np(G)$ and assume that the limit
\[
f(z):=\lim_{G \ni x\to z}f(x)\quad \text{exists}.
\]
Let $h$ be a solution of the $\K_{-\infty,f}(G)$-obstacle problem.
Then the fine limit
\[
\finelim_{G\ni x\to z} h(x)\quad \text{exists}.
\]
\end{cor}

\begin{proof} If $z$ is a regular point of $\partial G$, then
even the metric limit
\[
\lim_{G\ni x\to z} h(x)=f(z)
\]
exists, by (a) $\imp$ (g) in Theorem~6.11 in Bj\"orn--Bj\"orn~\cite{BB}
(or \cite[(a) $\imp$ (i) in Theorem~11.11]{BBbook}.
On the other hand, if $z$ is an
irregular boundary point of $G$, then $X\setminus G$ is thin at $z$ by
the sufficiency part of the Wiener criterion,
and $\Cp(\{z\})=0$ by the Kellogg property
(see Example~\ref{ex-only-fine-cont} for references).
Hence $G\cup \{z\}$ is finely open, $f\in \Np(G\cup\{z\})$, and the claim follows from
Theorem~\ref{thm-fine-lim}.
\end{proof}

In terms of Perron solutions on open sets, Corollary~\ref{cor-fine-lim}
yields the following consequence.
Here $Pf$ denotes the the Perron solution in $G$ with boundary data $f$,
see \cite[Section~10.3]{BBbook}. Recall that if $f \in C(\bdy G)$ then $f$ is resolutive
and thus $Pf$ exists, by Theorem~6.1 in
    Bj\"orn--Bj\"orn--Shan\-mu\-ga\-lin\-gam~\cite{BBS2}
(or \cite[Theorem~10.22]{BBbook}).

\begin{cor} \label{cor-fine-lim-Perron}
Let $G$ be a nonempty bounded open set with $\Cp(X \setm G)>0$
and let $f \in C(\bdy G)$.
Then the fine limit
\begin{equation} \label{eq-Perron}
    \finelim_{G\ni x\to z} Pf(x)
\end{equation}
exists for all $z\in\bdy G$.
\end{cor}

\begin{proof}
Let $\eps >0$.
Then there is $ \ft\in \Lipc(X)$ such that
\[
    f \le \ft \le f+\eps
   \quad \text{on } \bdy G.
\]
It follows from the definition of Perron solutions,
and the resolutivity of continuous functions, that
\[
  Pf \le P \ft \le Pf+\eps
   \quad \text{in  } G.
\]
Moreover, $P \ft$ is the continuous
solution of the $\K_{\ft,-\infty}(G)$-obstacle problem,  by
Theorem~5.1 in     Bj\"orn--Bj\"orn--Shan\-mu\-ga\-lin\-gam~\cite{BBS2}
  (or \cite[Theorem~10.12]{BBbook}).
Using Corollary~\ref{cor-fine-lim} and letting $\eps \to 0$
shows that  \eqref{eq-Perron} exists.
\end{proof}

\section{Removability}
\label{sect-remove}

\emph{In this section we assume that $U$ is a quasiopen set.}

\medskip

We conclude the paper by deducing some simple removability results.

\begin{lem} \label{lem-remove}
Let $E$ be a set with $\Cp(E)=0$ and $V=U \cup E$.
If $u \in \Npploc(V)$ is a
fine\/ \textup{(}super\/\textup{)}minimizer in $U$,
and $u$ is extended arbitrarily to $E$,
then $u$ is a fine\/ \textup{(}super\/\textup{)}minimizer in $V$.
\end{lem}

\begin{proof}
By Theorem~\ref{thm-quasiopen-char}, $V$ is quasiopen.
Let $\phi \in \Np_0(V)$.
Since $\phi=0$ q.e.\ in $X\setm U$, also $\phi \in \Np_0(U)$ and
the statement follows directly from Lemma~\ref{lem-char-1}.
\end{proof}

\begin{example} \label{ex-open-plus-pt}
(a) It is not enough to assume that $u \in \Npploc(U)$ in Lemma~\ref{lem-remove}.
Consider e.g.\ $p=2$, $U=B(0,1) \setm \{0\} \subset \R^2$  and
\[
u(x)=\log |x| \in N^{1,2}\loc(U)=N^{1,2}\fineloc(U),
\]
which is harmonic (and thus a fine minimizer) in $U$.
However $u$ has no extension
in $N^{1,2}\loc(V)=N^{1,2}\fineloc(V)$, with $V=B(0,1)$, and in particular no extension
as a fine superminimizer (i.e.\ as a superminimizer because $V$ is open),
even though $\Cp(V \setm U)=0$.

(b) Even if $U = \fineint V$, the assumption $u \in \Npploc(V)$ cannot be replaced by
$u \in \Npploc(U)$ in Lemma~\ref{lem-remove}.
Moreover, fine (super)minimizers on a quasiopen set $V$ can differ
from those on its fine interior $\fineint V$.

To see this, let $1<p<2$,
\[
U=(0,2)\times (-2,2) \quad \text{and} \quad V=U \cup \{(0,0)\}.
\]
This time,
\[
u(x)=|x|^{(p-2)/(p-1)} \in \Nploc(U)=\Npploc(U)
\]
is a fine minimizer (i.e.\ a minimizer) in the open set $U$,
see Example~7.47 in Heinonen--Kilpel\"ainen--Martio~\cite{HeKiMa}.

On the other hand, the set
\[
  W:=\{(x_1,x_2): 0<|x_2|<x_1<1\}  \Subset V
\]
is a \p-strict subset of both $U$ and $V$. This is easily seen by using
\[
h(x_1,x_2)=\eta(x) \min \biggl\{\frac{x_1}{|x_2|}, 1 \biggr\} \in \Np_0(U)
\]
with a suitable cutoff function $\eta$, see Example~5.7 in Bj\"orn--Bj\"orn~\cite{BB}
(or \cite[Example~11.10]{BBbook}).
However, $u \notin \Np(W)$ (and even $u \notin L^p(W)$ if $1 < p< \sqrt{2}$), so $u\notin\Npploc(V)$.
(Since $W \not \Subset U$, we still have $u \in \Npploc(U)$.)
\end{example}

\begin{cor}   \label{cor-remove-finelim}
Let $G$ be an open set and $V=G \setm E$, where $\Cp(E)=0$.
Assume that $u \in \Np(V)$ or $u \in \Nploc(G)$. Also let
\[
  u_*(x) = \fineliminf_{y \to x} u(y),\quad  \text{if } x\in V.
\]
\begin{enumerate}
\item \label{a-a}
If $u$ is a fine superminimizer in $V$,
then $u_*$ is finely continuous in $V$.
\item \label{a-b}
If $u$ is a fine minimizer in $V$,
then $u_*$ is continuous in $V$, with respect to the metric topology.
\end{enumerate}
\end{cor}

\begin{proof}
If $u \in \Np(V)$ then it follows from
Proposition~1.48 in~\cite{BBbook}
that $\ut \in \Np(G)$, where $\ut$ is any extension of $u$ to $G$.
Thus we can assume that $u \in \Nploc(G)$.
By Lemma~\ref{lem-remove}
and Corollary~\ref{cor-fine-min}, $u$ is a superminimizer in $G$.
It follows from Proposition~7.4 in Kinnunen--Martio~\cite{KiMa02}
(or \cite[Proposition~9.4]{BBbook}),
that $u$ has a
superharmonic representative $v$ such that $v=u$ q.e.\ in $G$.

In \ref{a-a}, $v$ is finely continuous in $G$, by
Bj\"orn~\cite[Theorem~4.4]{JB-pfine} or
Korte~\cite[Theorem~4.3]{korte08} (or \cite[Theorem~11.38]{BBbook}). 
In \ref{a-b}, $v$ is continuous in $G$,
by Kinnunen--Shanmugalingam~\cite[Proposition~3.3 and Theorem~5.2]{KiSh01}
(or \cite[Theorem~8.14]{BBbook}).

As $v=u$ q.e., we have $u_*=v_*=v$ in $V$, which proves the lemma.
\end{proof}

\section{Open problems}
\label{sect-open-prob}

Fine superminimizers and fine supersolutions  can be changed arbitrarily on sets of
capacity zero.
To fix a precise representative, in  potential theory
one usually studies pointwise
defined finely (super)harmonic functions with additional regularity properties,
as used in the proofs of Lemma~\ref{lem-fine-cont-supermin}
  and Corollary~\ref{cor-remove-finelim}.

In this paper, we do not
go further into making a definition
of finely (super)harmonic functions in metric spaces.
Even in the linear case, there have been several different
suggestions for such definitions in the literature,
see Luke\v{s}--Mal\'y--Zaj\'i\v{c}ek~\cite[Section~12.A and Remarks~12.1]{LuMaZa}.
Some definitions have been given in the nonlinear theory on $\R^n$,
but the theory is even less developed and
there are many open questions in this context.
A few of these are listed below.

\begin{openprobs} \label{openprobs-list}
  \quad % To get linebreak here
\begin{enumerate}
\renewcommand{\theenumi}{\textup{(\arabic{enumi})}}%
\renewcommand{\labelenumi}{\theenumi}%
\item \label{t-1}
  \emph{Is every finely superharmonic function finely continuous?}
  This is known in the linear case, see~\cite[Theorem~9.10]{Fug}
  and~\cite[Theorem~12.6]{LuMaZa}.
  In the nonlinear case, the best known result is
  Corollary~7.12 in Latvala~\cite{LatPhD}, which says that
  the finely superharmonic functions associated with
  the $n$-Laplacian on unweighted $\R^n$ are approximately
continuous.
\item
  \emph{Does every bounded  fine minimizer $u$ have
  a finely continuous representative $v$ such that $v=u$ q.e?}
  Even this special case of~\ref{t-1} is open in the nonlinear
  case.
  However, on unweighted $\R^n$, with $p=n \ge 2$,
  this fact  was shown   by Latvala~\cite[Lemma~7.15]{LatPhD}.
\item \label{t-3}
  \emph{If $u$ is a fine minimizer,
  is then $u_*$ finite everywhere?}
This seems to be open even in the linear case
on unweighted $\R^n$ (with  $p=2$),
since the connection between
  fine solutions ($=$ fine minimizers)
  and finely harmonic functions
  does not seem to have been touched upon in the linear literature.
\item \label{t-4}
  On unweighted $\R^n$, Latvala~\cite{Lat00}
  showed that $U\setminus E$ is a \p-fine domain 
  if $U$ is a \p-fine domain and $\Cp(E)=0$. As an application of this result a strong version of the
  minimum principle for finely superharmonic functions was obtained.
  \emph{We do not know if the corresponding fine connectedness result holds in our metric setting, or on weighted $\R^n$.}
\end{enumerate}
\end{openprobs}


\begin{thebibliography}{99}

\bibitem{ABclass} \art{Bj\"orn, A.}
         {A regularity classification of boundary points
           for \p-harmonic functions and quasiminimizers}
         {J. Math. Anal. Appl.} {338} {2008} {39--47}

\bibitem{BB} \art{Bj\"orn, A. \AND Bj\"orn, J.}
         {Boundary regularity for \p-harmonic functions and
           solutions of the obstacle problem}
         {J. Math. Soc. Japan} {58} {2006} {1211--1232}

\bibitem{BBbook} \book{Bj\"orn, A. \AND Bj\"orn, J.}
        {\it Nonlinear Potential Theory on Metric Spaces}
    {EMS Tracts in Mathematics {\bf 17},
        European Math. Soc., Z\"urich, 2011}

\bibitem{BBnonopen} \art{\auth{Bj\"orn}{A} \AND \auth{Bj\"orn}{J}}	
	{Obstacle and Dirichlet problems on arbitrary nonopen sets
          in metric spaces, and fine topology}
        {Rev. Mat. Iberoam.} {31} {2015} {161--214}

\bibitem{BBfusco} \art{\auth{Bj\"orn}{A} \AND \auth{Bj\"orn}{J}}	
  {A uniqueness result for functions with zero fine gradient
    on quasiconnected sets}
  {Ann. Sc. Norm. Super. Pisa Cl. Sci.} {21} {2020} {293--301}

\bibitem{BBLat1} \art{\auth{Bj\"orn}{A}, \auth{Bj\"orn}{J}  \AND
    \auth{Latvala}{V}}
        {The weak Cartan property for the \p-fine topology on metric spaces}
	{Indiana Univ. Math. J.} {64} {2015} {915--941}

\bibitem{BBLat3} \art{\auth{Bj\"orn}{A}, \auth{Bj\"orn}{J}  \AND
    \auth{Latvala}{V}}
        {Sobolev spaces,  fine gradients and quasicontinuity on quasiopen sets
in $\R^n$ and metric spaces}
	{Ann. Acad. Sci. Fenn. Math.} {41} {2016} {551--560}

\bibitem{BBLat2} \art{\auth{Bj\"orn}{A}, \auth{Bj\"orn}{J}  \AND
    \auth{Latvala}{V}}
        {The Cartan, Choquet and Kellogg properties of the
        fine topology on metric spaces}
	{J. Anal. Math.} {135} {2018} {59--83}

\bibitem{BBMaly} \art{\auth{Bj\"orn}{A}, \auth{Bj\"orn}{J}  \AND
    \auth{Mal\'y}{J}}
        {Quasiopen and \p-path open sets, and characterizations of quasicontinuity}
	{Potential Anal.} {46} {2017} {181--199}

\bibitem{BBS} \art{Bj\"orn, A., Bj\"orn, J. \AND Shanmugalingam, N.}
         {The Dirichlet problem for \p-harmonic functions on metric spaces}
         {J. Reine Angew. Math.} {556} {2003} {173--203}

\bibitem{BBS2} \art{Bj\"orn, A., Bj\"orn, J. \AND Shan\-mu\-ga\-lin\-gam, N.}
         {The Perron method for \p-harmonic functions}
         {J. Differential Equations} {195} {2003} {398--429}

\bibitem{BBS5} \art{\auth{Bj\"orn}{A}, \auth{Bj\"orn}{J}
       \AND \auth{Shan\-mu\-ga\-lin\-gam}{N}}
        {Quasicontinuity of Newton--Sobolev functions and density of Lipschitz
        functions on metric spaces}
        {Houston J. Math.} {34} {2008} {1197--1211}

\bibitem{JB-Matsue} \artin{\auth{Bj\"orn}{J}}
        {Wiener criterion for Cheeger \p-harmonic
        functions on metric spaces}
        {{\it Potential Theory in Matsue},
        Advanced Studies in Pure Mathematics {\bf 44}, pp. 103--115,
        Mathematical Society of Japan, Tokyo, 2006}

\bibitem{JB-pfine}  \art{Bj\"orn, J.} {Fine continuity on metric spaces}
         {Manuscripta Math.} {125} {2008} {369--381}

\bibitem{BMS}  \art{Bj\"orn, J., MacManus, P. \AND Shanmugalingam, N.}
        {Fat sets and pointwise boundary estimates for \p-harmonic
        functions in metric spaces}
        {J. Anal. Math.}{85}{2001}{339--369}

\bibitem{BBV}  \art{Bucur, D., Buttazzo, G. \AND Velichkov, B.}
        {Spectral optimization problems with internal constraint}
        {Ann. Inst. H. Poincar\'e Anal. Non Lin\'eaire}{30}{2013}{477--495}

\bibitem{buttazzo-dalMaso} \art{\auth{Buttazzo}{G} \AND \auth{Dal Maso}{G}}
     {An existence result for a class of shape optimization problems}
     {Arch. Ration. Mech. Anal.} {122} {1993} {183--195}

\bibitem{buttazzo-shr} \art{\auth{Buttazzo}{G} \AND \auth{Shrivastava}{H}}
  {Optimal shapes for general integral functionals}
  {Ann. H. Lebesgue} {3} {2020} {261--272}

\bibitem{cartan46} \art{\auth{Cartan}{H}}
     {Th\'eorie g\'en\'erale du balayage en potentiel newtonien}
     {Ann. Univ. Grenoble. Sect. Sci. Math. Phys.} {22} {1946} {221--280}

\bibitem{Fugl71} \art{\auth{Fuglede}{B}}
   {The quasi topology associated with a countably subadditive set function}
   {Ann. Inst. Fourier\/ \textup{(}Grenoble\/\textup{)}}
   {21{\rm:1}}{1971}{123--169}

\bibitem{Fug} \book{\auth{Fuglede}{B}}
         {Finely Harmonic Functions}
         {Springer, Berlin--New York, 1972}

\bibitem{Fug74} \art{Fuglede, B.}
        {Fonctions harmoniques et fonctions finement harmoniques}
        {Ann. Inst. Fourier\/ \textup{(}Grenoble\/\textup{)}}
        {24{\rm:4}}{1974}{77--91}

% Implicitly cited
\bibitem{Fug75:1} \art{Fuglede, B.}
        {Asymptotic paths for subharmonic functions}
        {Math. Ann.}
        {213}{1975}{261--274}

% Implicitly cited
\bibitem{Fug75:2} \art{Fuglede, B.}
        {Sur la fonction de Green pour un domaine fin}
        {Ann. Inst. Fourier\/ \textup{(}Grenoble\/\textup{)}}
        {25{\rm:3-4}}{1975}{201--206}

% Implicitly cited
\bibitem{Fug81} \art{Fuglede, B.}
        {Sur les fonctions finement holomorphes}
        {Ann. Inst. Fourier\/ \textup{(}Grenoble\/\textup{)}}
        {31{\rm:4}} {1981} {vii, 57--88}

\bibitem{Fug90} \art{Fuglede, B.}
        {On the mean value property of finely harmonic and finely hyperharmonic functions}
         {Aequationes Math.}{39}{1990}{198--203}

\bibitem{FuscoMZ} \art{\auth{Fusco}{N}, \auth{Mukherjee}{S}
        \AND \auth{Zhang}{Y. R.-Y}}
      {A variational characterisation of the second eigenvalue of the
        \p-Laplacian on quasi open sets}
      {Proc. Lond. Math. Soc.} {119} {2019} {579--612}

\bibitem{HeKiMa} \book{\auth{Heinonen}{J},
	\auth{Kilpel\"ainen}{T}
	\AND \auth{Martio}{O}}
        {Nonlinear Potential Theory of Degenerate Elliptic Equations}
        {2nd ed., Dover, Mineola, NY, 2006}

\bibitem{HKST} \book{\auth{Heinonen}{J}, \auth{Koskela}{P},
	\auth{Shanmugalingam}{N} \AND \auth{Tyson}{J. T}}
       {Sobolev Spaces on Metric Measure Spaces}
	{New Mathematical Monographs {\bf 27}, Cambridge Univ. Press,
        Cambridge, 2015}

\bibitem{KiMa92} \art{Kilpel\"ainen, T. \AND Mal\'y, J.}
        {Supersolutions to degenerate elliptic equation on quasi open sets}
        {Comm. Partial Differential Equations}
        {17} {1992} {371--405}

\bibitem{KiMa94} \art{Kilpel\"ainen, T. \AND Mal\'y, J.}
{The Wiener test and potential estimates for quasilinear elliptic equations}
{Acta Math.}{172} {1994}{137--161}

\bibitem{KiMa02} \art{Kinnunen, J. \AND Martio, O.}
         {Nonlinear potential theory on metric spaces}
         {Illinois Math. J.} {46} {2002} {857--883}

\bibitem{KiSh01} \art{Kinnunen, J. \AND Shanmugalingam, N.}
         {Regularity of quasi-minimizers on metric spaces}
         {Manuscripta Math.} {105} {2001} {401--423}

\bibitem{korte08} \art{Korte, R.}
         {A Caccioppoli estimate and fine continuity for superminimizers on metric spaces}
         {Ann. Acad. Sci. Fenn. Math.} {33} {2008} {597--604}

\bibitem{Lahti17} \art{\auth{Lahti}{P}}
  {A notion of fine continuity for BV functions on metric spaces}
  {Potential Anal.} {46} {2017} {279--294}

% Implicitly cited
\bibitem{Lahti18} \art{\auth{Lahti}{P}}
  {A new Cartan-type property and strict quasicoverings when $p=1$
    in metric spaces}
  {Ann. Acad. Sci. Fenn. Math.} {43} {2018} {1027--1043}

% Implicitly cited
\bibitem{Lahti19} \art{\auth{Lahti}{P}}
  {The Choquet and Kellogg properties for the fine topology when
    $p=1$ in metric spaces}
  {J. Math. Pures Appl.} {126}{2019} {195--213}

\bibitem{Lahti20} \art{\auth{Lahti}{P}}
  {Superminimizers and a weak Cartan property for $p=1$ in metric spaces}
  {J. Anal. Math.} {140} {2020} {55--87}


\bibitem{LatPhD} \book{\auth{Latvala}{V}}
        {Finely Superharmonic Functions of Degenerate Elliptic Equations}
        {Ann. Acad. Sci. Fenn. Ser. A I Math. Dissertationes {\bf 96}
        {(1994)}}

\bibitem{Lat00} \art{Latvala, V.}
        {A theorem on fine connectedness}
        {Potential Anal.} {12} {2000} {221--232}

\bibitem{Lind-Mar} \art{\auth{Lindqvist}{P} \AND \auth{Martio}{O}}
        {Two theorems of N. Wiener for solutions of quasilinear
          elliptic equations}
        {Acta Math.}{155}{1985}{153--171}

\bibitem{LuMa} \art{\auth{Luke\v{s}}{J} \AND \auth{Mal\'y}{J}}
       {Fine hyperharmonicity without Axiom D}
       {Math. Ann.} {261} {1982} {299--306}

\bibitem{LuMaZa} \book{\auth{Luke\v{s}}{J}, \auth{Mal\'y}{J} \AND
         \auth{Zaj\'i\v{c}ek}{L}}
         {Fine Topology Methods in Real Analysis and Potential Theory}
         {Springer, Berlin--Heidelberg, 1986}

\bibitem{Lyons80:1} \art{Lyons, T.}
        {Finely holomorphic functions}
        {J. Funct. Anal.}{37}{1980}{1--18}

\bibitem{Lyons80:2} \art{Lyons, T.}
        {A theorem in fine potential theory and applications to finely holomorphic functions}
        {J. Funct. Anal.}{37}{1980}{19--26}

\bibitem{MZ} \book{Mal\'y, J. \AND Ziemer, W. P.}
{Fine Regularity of Solutions of Elliptic Partial Differential Equations}
{Amer. Math. Soc., Providence, RI, 1997}

\bibitem{Maz70} \art{\auth{Maz{\cprime}ya}{V. G}}
        {On the continuity at a boundary point of solutions of quasi-linear
        elliptic equations}
        {Vestnik Leningrad. Univ. Mat. Mekh. Astronom.}
        {25{\rm:13}} {1970} {42--55  (Russian)}
        English transl.: {\it Vestnik Leningrad Univ. Math.}
        {\bf 3} (1976), 225--242.

\bibitem{Mikkonen} \book{Mikkonen, P.}
    {On the Wolff Potential and Quasilinear Elliptic Equations Involving
    Measures}
    {Ann. Acad. Sci. Fenn. Math. Diss. {\bf 104} (1996)}

\bibitem{Sh-harm} \art{Shanmugalingam, N.}
        {Harmonic functions on metric spaces}
        {Illinois J. Math.}{45}{2001}{1021--1050}

\bibitem{wiener} \art{\auth{Wiener}{N}}
        {The Dirichlet problem}
        {J. Math. Phys.}{3}{1924}{127--146}


\end{thebibliography}
\end{document}